\documentclass{amsart}
\usepackage{enumerate}
\usepackage{amssymb}
\usepackage{xcolor}
\usepackage{mathrsfs}
\usepackage{hyperref}
\usepackage{graphicx}

\numberwithin{equation}{section}
\newtheorem{thm}{Theorem}[section]
\newtheorem{lem}[thm]{Lemma}

\newtheorem{prop}[thm]{Proposition}
\newtheorem{cor}[thm]{Corollary}

\newtheorem{ex}[thm]{Example}

\newtheorem{rem}[thm]{Remark}
\newtheorem{defi}[thm]{Definition}

\newcommand{\bD}{\mathbf{D}}

\newcommand{\calculations}[1]{}

\begin{document}

\title[Local invariants of conformally deformed non-commutative tori II]{Local invariants of conformally deformed non-commutative tori II: multiple operator integrals}

\author{Teun van Nuland}
\address{University of New South Wales, Kensington, NSW, 2052, Australia}
\email{teunvn@gmail.com}
\author{Fedor Sukochev}
\address{University of New South Wales, Kensington, NSW, 2052, Australia}
\email{f.sukochev@unsw.edu.au}
\author{Dmitriy Zanin}
\address{University of New South Wales, Kensington, NSW, 2052, Australia}
\email{d.zanin@unsw.edu.au}

\begin{abstract} We explicitly compute the local invariants (heat kernel coefficients) of a conformally deformed non-commutative $d$-torus using multiple operator integrals. We derive a recursive formula that easily produces an explicit expression for the local invariants of any order $k$ and in any dimension $d$. Our recursive formula can conveniently produce all formulas related to the modular operator, which before were obtained in incremental steps for $d\in\{2,4\}$ and $k\in\{0,2,4\}$. We exemplify this by writing down some known ($k=2$, $d=2$) and some novel ($k=2$, $d\geq 3$) formulas in the modular operator.
\end{abstract}


\maketitle

\section{Introduction}

Studying the heat trace expansion on a non-commutative manifold, and computing the respective local invariants (i.e., the non-commutative heat kernel coefficients), is vital for two reasons. Firstly, the heat kernel coefficients play a major role in quantum field theory (cf. \cite{Vassilevich}), and if space turns out to be non-commutative at small scale, these coefficients will need to be generalised. 
Secondly, the local invariants allow to extract geometric information from the spectrum of a Laplace-type operator, and they are therefore good starting points to extend geometric concepts to the setting of non-commutative geometry.


We shall focus on the non-commutative $d$-tori $\mathbb T^d_\theta$, as they are prime examples of non-commutative spaces. A benefit of these examples is that they have clear-cut non-commutative analogues $C^\infty(\mathbb T^d_\theta)$ and $L_\infty(\mathbb T^d_\theta)$ of the commutative algebras $C^\infty(\mathbb T^d)$ and $L_\infty(\mathbb T^d)$, together with a faithful representation $\lambda_l$ on the Hilbert space $L_2(\mathbb T^d_\theta)$, which is an analogue of $L_2(\mathbb T^d)$, and a trace $\tau:L_\infty(\mathbb T^d_\theta)\to\mathbb C$, which is an analogue of integration; see the definitions in Section \ref{sct:NCT}.

The local invariants $I_k(P)$ of an operator $P$ acting in $L_2(\mathbb T^d_\theta)$ are the unique coefficients occurring in the heat trace expansion, which is the asymptotic expansion
\begin{align}\label{eq:asymptotic expansion}
{\rm Tr}(\lambda_l(y)e^{-tP})\sim \sum_{\substack{k\geq0\\ k=0{\rm mod}2}}t^{\frac{k-d}{2}}\tau(y\, I_k(P)),\quad t\downarrow0\qquad (y\in L_\infty(\mathbb T^d_\theta)).
\end{align}
In \cite{SZ-local} it was shown that this expansion exists if (and in particular $e^{-tP}$ is trace class if) $P$ is self-adjoint and of the form
\begin{align}\label{eq:P}
P=\lambda_l(x)\Delta+\sum_{i=1}^d\lambda_l(a_i)D_i+\lambda_l(a)\quad\text{for some}\quad x,a_i,a\in C^{\infty}(\mathbb{T}^d_{\theta}),
\end{align}
with $x$ positive and invertible. Here, $\Delta=\sum_{i=1}^d D_i^2$ and $D_i$ is the $i^{\text{th}}$ directional derivative which, again, is defined in Section \ref{sct:NCT}. In this generalised sense, $P$ is a strongly elliptic differential operator. Besides the existence of the asymptotic expansion, \cite{SZ-local} shows that $I_k(P)\in C^\infty(\mathbb T^d_\theta)$, so $y\mapsto \tau(y\, I_k(P))$ can informally be thought of as `integration' against a smooth `function' $I_k(P)$. The goal of this paper is to explicitly compute $I_k(P)$ for this class of $P$ (those of the form \eqref{eq:P} with $x$ positive and invertible). 

One motivation for us to consider this class of $P$ is the vibrant research program that surrounds the local invariants of the so-called conformally deformed non-commutative torus, a research program that was initiated by the papers \cite{Cohen-Connes,Connes-Moscovici,Connes-Tretkoff}.
The classical limit of the non-commutative torus is simply the flat torus $\mathbb T^d$, which holds no interesting geometry, and likewise $I_k(\Delta)$ is trivial. Geometric non-triviality is added to the torus in \cite{Cohen-Connes} by an adjustment analogous to a conformal scaling of the metric, and the result is called the conformally deformed non-commutative torus. For our purposes, we can capture this conformal scaling (see, e.g., \cite{Connes-Moscovici}) by replacing the Laplacian by operators $P$ of the form \eqref{eq:P} and keeping the same Hilbert space, algebra, and representation.


A main goal in the research program mentioned above is to express the local invariants $I_k(P)$ as closed formulas involving functional calculus applied to the modular operator defining the conformal scaling, as done in \cite{Connes-Moscovici,FaKh1} for $k=2$, $d=2$.

New functions acting by this `modular functional calculus' were found for $d=4,k=2$ in \cite{FaKh} and for $d=2,k=4$ in \cite{CoFa} and to understand and sometimes simplify the vast calculations in these papers, significant progress has been made in \cite{Fa,Lesch,Lesch-Moscovici,Liu}. In this context, various geometric notions were lifted to the non-commutative setting in \cite{Connes-Tretkoff,FaKh1,GloGorKh,Ponge1,Ponge2,PongeHa,Liu1} \textit{et cetera}, 
giving further motivation for the calculation of the local invariants in terms of modular functional calculus, but not yet extending that calculation to higher dimension or higher order.
In most of the papers mentioned above, the local invariants are obtained by means of the zeta function associated with $P$.

Notably, in \cite{IM1,IM2,IM3}, Iochum and Masson took a very different approach, and computed the local invariants of matrix-valued differential operators acting on bundles over manifolds, extending the formulas and algorithms known for computing the Seeley-deWitt coefficients as also done by Gilkey in \cite{Gilkey} (see \cite{Vassilevich} for an overview) to the matrix-valued case. 
This approach is more directly applicable in physics, and further extending their formulas to non-commutative manifolds is likely a worthwhile pursuit.
Indeed, the confrontation of the spectral standard model \cite{Chamseddine-van Suijlekom} with physics, which so far looks promising (cf. \cite{Connes-Chamseddine}), relies precisely on this heat trace expansion for bundles over manifolds (cf. \cite{van Suijlekom}), while in non-commutative quantum field theories like \cite{Grosse-Wulkenhaar} the underlying space itself is non-commutative, and bares resemblance to a non-commutative torus. Similarly, the non-commutative $d$-torus arises from matrix theory compactification, as explained in \cite{CDS,KS}.



In this paper, we give a way to compute $I_k(P)$ explicitly for every $d\in\mathbb N_{\geq2}$ and every $k\in\mathbb Z_{\geq0}$, by making use of the full power of multiple operator integration theory, and using the description of $I_k(P)$ that in \cite{SZ-local} led to the existence of the asymptotic expansion \eqref{eq:asymptotic expansion}. Instead of the abstract modular functional calculus employed by Connes and others, we use the established framework of multiple operator integrals (a short introduction for the non-affiliate is given in Section \ref{sct:MOI}).  Our approach is more similar to the (almost) commutative approach of \cite{IM1,IM2,IM3}, which (as argued in the paragraph above) makes it ideal for physical applications.

Our main result is a compact expression for $I_k(P)$ that involves a simple recursive rule. When a specific $k$ is chosen, this expression can be recursively expanded, and the resulting expressing for $I_k(P)$ is a sum of explicit multiple operator integrals (that in many cases can be computed algebraically, as in Remark \ref{rem:algebraic MOI}). The amount of terms blows up rapidly (1 term for $k=0$, 13 terms for $k=2$, 1046 terms for $k=4$, 140845 terms for $k=6$, \textit{et cetera}). 

We show how one can straightforwardly obtain all functions acting by modular functional calculus from the just-mentioned expressions in terms of multiple operator integrals. Thus we conclude a list \cite{CoFa,Connes-Moscovici,FaKh1,FaKh} of advancements in which such functions were found for slowly increasing $k$ and $d$. Our approach yields a substantial insight into the structure and appearance of such functions in \cite{CoFa,Connes-Moscovici,FaKh1,FaKh}, namely, as the results of a recursive procedure that starts with a simple expression, and increasingly jumbles up the result in each step. In particular, our recursive structure explains the appearance of divided differences in these functions as a result of a commutator rule for multiple operator integrals that is also central to \cite{vNvS}, and hence (through the arguments of \cite{Lesch,Liu2022}) sheds light on the functional relations of \cite{CoFa,Connes-Moscovici}.

Passing from multiple operator integrals to modular functional calculus is quite simple, which we exemplify by producing some known and some novel modular formulas.

\textbf{This paper is structured as follows.} After introducing the beautiful subject of multiple operator integration and fixing our notation in Section \ref{sct:prelim}, we give a comprehensive summary of our main results and applications in Section \ref{sct:main results}. Section \ref{prelim section} contains some groundwork. After that, the proof of our main theorem, Theorem \ref{thm: recursive formula}, will span the whole of Sections \ref{sct:initial}, \ref{sct:Integral formula S and T}, and \ref{sct:pf main thm}. In Section \ref{sct:k=2} we apply our main result to the case $k=2$ and prove Corollary \ref{cor:main thm k=2} and Theorem \ref{cor:main thm k=d=2}. The connection with modular functional calculus is made in Section \ref{sct:Connes-Moscovici}, and the connection with the commutative case is made in Section \ref{sct:Iochum-Masson}. Appendix \ref{appendix A} comments on the accompanying python program.

\textbf{Acknowledgements.~} We heartily thank Bruno Iochum, Matthias Lesch, and Adam Rennie for very illuminating discussions. We are grateful to Christiaan van de Ven for pointing us to \cite{Folland} and to Yerlan Nessipbayev for checking part of our computations. T. v.\hspace*{-0.5pt} N. was supported in part by NSF CAREER grant DMS-1554456, as well as in part by ARC grant FL17010005. F. S. was supported by ARC grant FL17010005 and ARC grant DP230100434. D. Z. was supported by ARC grant DP230100434.

\section{Preliminaries}\label{sct:prelim}

We let $\mathbb N=\{1,2,\ldots\}$, $\mathbb Z_+=\mathbb Z_{\geq 0}$, $\mathbb T:=\mathbb R/\mathbb Z$, $\mathbb R_-=(-\infty,0)$. 
On a suitable set $X$ we let $C^\infty(X),L_1(X),L_2(X)$ denote the smooth, Lebesgue integrable, and square-integrable functions, respectively. For $n\in\mathbb Z^m$ we write $|n|_1:=\sum_{i=1}^m|n_i|$. The bounded operators on a Hilbert space $\mathcal H$ are denoted $\mathcal B(\mathcal H)$.

\subsection{Multiple operator integrals}\label{sct:MOI}
\subsubsection{Introduction}

The role that multiple operator integrals play for local invariants has never been clearly spelled out (and was but mentioned in \cite{Liu2022}).
However, when computing local invariants, one frequently (e.g. in \cite{FaKh1,IM1,IM2,IM3,Lesch,Liu}) encounters integrals that look roughly similar to, for example,
\begin{align}\label{eq:MOI 1}
\int_0^1 \int_0^{s_1} e^{(s_1-1)x}\,V_1\,e^{(s_2-s_1)x}\,V_2\,e^{-s_2x}\, ds_2 \, ds_1;\\
\label{eq:MOI 2}
\int_{\mathbb R} \, \frac{1}{x+i\lambda}\,V_1\,\frac{1}{x+i\lambda}\,V_2\,\frac{1}{x+i\lambda}\frac{e^{i\lambda}}{2\pi} \,d\lambda;
\end{align}
or any other integral over an alternating product of bounded operators $V_1,V_2$ and functions of a self-adjoint operator $x$. These are in fact special cases of multiple operator integrals, namely, integrals of the form \eqref{eq:MOI} in the following definition.
\begin{defi}\label{def:MOI}
Let $n\in\mathbb N$, and let $x$ be a (possibly unbounded) self-adjoint operator in a separable Hilbert space $\mathcal H$. Let $\phi:\mathbb R^{n+1}\to\mathbb C$ be given by
\begin{align}\label{eq:sep of variables phi}
\phi(\alpha_0,\ldots,\alpha_n)=\int_\Omega a_0(\alpha_0,\lambda)a_1(\alpha_1,\lambda)\cdots a_n(\alpha_n,\lambda)\,d\lambda,
\end{align}
 for bounded measurable functions $a_0,\ldots,a_n:\mathbb R\times\Omega\to\mathbb C$ and a finite measure space $(\Omega,\lambda)$. The multiple operator integral is the multilinear function $T_\phi^x:\mathcal B(\mathcal H)^{\times n}\to \mathcal B(\mathcal H)$ defined by
\begin{align}\label{eq:MOI}
T^x_\phi(V_1,\ldots,V_n):=\int_\Omega a_0(x,\lambda)V_1a_1(x,\lambda)\cdots V_na_n(x,\lambda)\,d\lambda,
\end{align}
for $V_1,\ldots,V_n\in\mathcal B(\mathcal H)$.
\end{defi}
This definition was given in \cite[Definition 4.1]{ACDS} and \cite{Peller}, and it is a simple but crucial result of \cite[Lemma 4.3]{ACDS} and \cite[Lemma 3.1]{Peller} that $T^x_\phi$ only depends on the functions $a_0,\ldots,a_n$ through $\phi$, as the notation suggests.
Indeed, under reasonable assumptions, \eqref{eq:MOI 1} equals \eqref{eq:MOI 2}. This explains why the literature sometimes contains different procedures to calculate the same thing.
Providing an elegant unified picture is not the only purpose of multiple operator integration (and we ensure the critical reader that we are not merely casting known results into new notation). The theory of multiple operator integration provides extremely strong results on the analytical properties of integrals like \eqref{eq:MOI 1} and \eqref{eq:MOI 2}, and moreover, the formalism often leads to completely new results or extensive generalisations of known ones. (See \cite{Skripka-Tomskova} for an overview of theory and applications.)
\subsubsection{Basic results on multiple operator integrals}
If $n=0$, then Definition \ref{def:MOI} recovers functional calculus: $$T^x_\phi()=\phi(x).$$ Moreover, if $V_1,\ldots,V_n$ commute with $x$ then $T^x_\phi(V_1,\ldots,V_n)$ by definition reduces to
$\phi(x,\ldots,x)V_1\cdots V_n.$
 This paper deals exclusively with bounded $x\in \mathcal B(\mathcal H)$, so let us assume this from now on.
We remark that $\phi\mapsto T^x_\phi$ factors through $\phi\mapsto \phi|_{{\rm spec}(x)^{n+1}}$. The function $\phi$ is called the symbol of $T^x_\phi$. The symbols we encounter most are \textit{divided differences} $\phi=f^{[n]}$ of some $f\in C^n(\mathbb R)$. By definition, and, subsequently, by induction, we have
\begin{align}
\label{eq:first id dd}f^{[0]}(\alpha_0):=&~f(\alpha_0);\\
\label{eq:second id dd}f^{[n]}(\alpha_0,\ldots,\alpha_n):=&\frac{f^{[n-1]}(\alpha_0,\alpha_1,\ldots,\alpha_{n-1})-f^{[n-1]}(\alpha_1,\ldots,\alpha_n)}{\alpha_0-\alpha_n}\\
\label{eq:third id dd}=&~\int_{\mathbb R}\int_{S^n} \widehat{f^{(n)}}(t)e^{it\lambda_0\alpha_0}\cdots e^{it\lambda_n\alpha_n}\,d\lambda\,dt,
\end{align}
where $\lambda$ is the flat measure on the simplex $S^n=\{\lambda\in\mathbb R_{\geq0}^{n+1}:~\sum_{j=0}^n\lambda_j=1\}$ with $\lambda(S^n)=1/n!$, and we have assumed for simplicity that the Fourier transform $\widehat{f^{(n)}}$ exists in $L_1(\mathbb R)$, and that $\alpha_0\neq\alpha_n$. It is well known that $f^{[n]}$ actually extends to $f^{[n]}\in C(\mathbb R^{n+1})$, satisfies $f^{[n]}(\alpha,\ldots,\alpha)=\frac{1}{n!}f^{(n)}(\alpha)$, and is invariant under permutations of its arguments, for all $f\in C^n(\mathbb R)$.
By comparing \eqref{eq:third id dd} to \eqref{eq:sep of variables phi}, we notice that $T^x_{f^{[n]}}$ is defined whenever $f\in C^n(\mathbb R)$ and $\widehat{f^{(n)}}\in L_1(\mathbb R)$. By \eqref{eq:first id dd} and \eqref{eq:second id dd}, $f\mapsto T^x_{f^{[n]}}$ factors through $f\mapsto f|_I$, where $I$ is any neighbourhood of ${\rm spec}(x)$, and hence $T^x_{f^{[n]}}$ is defined whenever $f\in C^\infty(\mathbb R)$, as $x$ is assumed bounded. More generally, $T^x_\phi$ is defined whenever $\phi\in C^\infty(\mathbb R^{n+1})$.
Some important identities in this case are (cf. \cite[Lemma 14]{vNvS})
\begin{align}
&T^x_{f^{[n]}}(V_1,\ldots,V_j,yV_{j+1},\ldots,V_n)-T^x_{f^{[n]}}(V_1,\ldots,V_jy,V_{j+1},\ldots,V_n)\nonumber\\
&=T^x_{f^{[n+1]}}(V_1,\ldots,V_j,[x,y],V_{j+1},\ldots,V_n);\nonumber\\
&T^x_{f^{[n]}}(yV_1,\ldots,V_n)-yT^x_{f^{[n]}}(V_1,\ldots,V_n)=T^x_{f^{[n+1]}}([x,y],V_1,\ldots,V_n);\nonumber\\
&T^x_{f^{[n]}}(V_1,\ldots,V_n)y-T^x_{f^{[n]}}(V_1,\ldots,V_ny)=T^x_{f^{[n+1]}}(V_1,\ldots,V_n,[x,y]);\label{eq:MOI comm 3}\\
&f(x)y-yf(x)=T^x_{f^{[0]}}()y-yT^x_{f^{[0]}}()=T^x_{f^{[1]}}([x,y]),\nonumber
\end{align}
for $x,y,V_j\in \mathcal B(\mathcal H)$.
In this paper we often see $f=F_{k,d}$, where $F_{k,d}$ is our notation for the $\frac{k}{2}^{\text{th}}$ order primitive of $\alpha\mapsto\alpha^{-\frac{d}{2}}$ on $(0,\infty)$, and ${\rm spec}(x)\subseteq (0,\infty)$. In the literature, the dimension $d$ of the (non-commutative) space is often even and the order $k$ (appearing in $I_k(P)$) is often small. In these abundant cases the multiple operator integral is extremely explicit:
\begin{rem}\label{rem:algebraic MOI}
Suppose that $d$ is even and that $k<d$. Then $F_{k,d}$ is an integer power function, so for every $n\in\mathbb N$ there exists a finite $I \subseteq\mathbb Z^{n+1}$ and some constants $c_i$ such that, for all $\alpha_j>0$,
$$F_{k,d}^{[n]}(\alpha_0,\ldots,\alpha_n)=\sum_{i\in I}c_i\alpha_0^{i_0}\cdots\alpha_n^{i_n}.$$
As a consequence, the multiple operator integral is a purely algebraic expression,
$$T^x_{F_{k,d}^{[n]}}(V_1,\ldots,V_n)=\sum_{i\in I}c_ix^{i_0}V_1x^{i_1}\cdots V_nx^{i_n},$$
for all $V_j\in\mathcal B(\mathcal H)$ and positive invertible $x\in\mathcal B(\mathcal H)$.
\end{rem}

\subsection{Non-commutative torus}\label{sct:NCT}
Regarding the non-commutative torus, we use the definitions of \cite{SZ-local}, which we only briefly recall in this section. We omit the proofs of the folklore assertions below, some of which can be found in \cite[\textsection 2]{SZ-local}. 

For any $d\in\mathbb N_{\geq2}$, we let $\theta\in M_d(\mathbb R^d)$ be an antisymmetric matrix. Let $A_\theta$ be the unital *-algebra generated by formal symbols $U_1,\ldots,U_d$ satisfying $U_k^*U_k=U_kU_k^*=1$ and $U_kU_l=e^{2\pi i\theta_{kl}}U_lU_k$, and write $U^n:=U_1^{n_1}\cdots U_d^{n_d}$ for all $n\in\mathbb Z^d$. We define a linear function $\tau:A_\theta\to\mathbb C$ by $\tau(\sum_n c_n U^n):=c_0$. We let $L_2(\mathbb T^d_\theta)$ be the completion of $A_\theta$ in the norm $\|a\|:=\langle a,a\rangle^{\frac12}$ defined by the (nondegenerate) inner product $\langle a,b\rangle:=\tau(a^*b)$, which makes $\mathcal H:=L_2(\mathbb T^d_\theta)$ a separable Hilbert space. We define $D_k(\sum_n c_n U^n):=\sum_n c_n n_kU^n$ on $A_\theta$, and let $C^\infty(\mathbb T^d_\theta)$ be the completion of $A_\theta$ in the (Fr\'echet) seminorms $a\mapsto \|D^\alpha a\|$, $\alpha\in\mathbb Z^d_+$, where $D^\alpha:=D_1^{\alpha_1}\cdots D_d^{\alpha_d}$. Each $D^\alpha$ extends to a self-adjoint operator densely defined in $L_2(\mathbb T^d_\theta)$, with $C^\infty(\mathbb T^d_\theta)=\cap_\alpha {\rm dom}\,D^\alpha$. Hence, $C^\infty(\mathbb T^d_\theta)$ is stable under holomorphic functional calculus. We represent $A_\theta$ on $L_2(\mathbb T^d_\theta)$ by $\lambda_l(a)b:=ab$, and denote by $L_\infty(\mathbb T^d_\theta)$ the corresponding weak closure of $A_\theta$, a von Neumann algebra with operator norm denoted $\|\cdot\|_\infty$. We identify $C^\infty(\mathbb T^d_\theta)\subseteq L_\infty(\mathbb T^d_\theta)$ and $L_2(\mathbb T^d_\theta)\subseteq L_\infty(\mathbb T^d_\theta)$. Both $\tau$ and $\lambda_l$ extend continuously to $L_\infty(\mathbb T^d_\theta)$, giving a faithful tracial state $\tau:L_\infty(\mathbb T^d_\theta)\to \mathbb C$ and a faithful representation (injective *-homomorphism) $\lambda_l:L_\infty(\mathbb T^d_\theta)\to\mathcal B(L_2(\mathbb T^d_\theta))$.

%

\section{Summary of main results}\label{sct:main results}

Before stating our main result, we shall introduce the recursive structure that lies at its core.

\subsection{Recursive structure}

We let $\mathbf D_1,\ldots,\mathbf D_d$ be the formal symbols of the polynomial algebra $\mathbb C[\mathbf D_1,\ldots,\mathbf D_d]$, i.e., we impose only the relation $\mathbf D_i\mathbf D_j=\mathbf D_j\mathbf D_i$ for all $i,j\in\{1,\ldots,d\}$. We write $\mathbf D^\alpha:=\mathbf D_1^{\alpha_1}\cdots\mathbf D_d^{\alpha_d}$ for all $\alpha\in\mathbb Z_+^d$.
We then define the free left $C^\infty(\mathbb T^d_\theta)$-module 
$$\mathcal{X}:= {\rm span}\big\{b\mathbf{D}^\alpha:~b\in C^\infty(\mathbb{T}^d_\theta),\alpha\in\mathbb Z_+^d\big\},$$ generated by the set of formal symbols $\{\mathbf D^\alpha:~\alpha\in\mathbb Z_+^d\}$, and refer to elements of $\mathcal{X}$ as formal differential operators. We identify $C^\infty(\mathbb T^d_\theta)\subseteq\mathcal{X}$ by identifying $b=b\mathbf D^0$.

\begin{defi}\label{def:bT}
Let $x\in C^\infty(\mathbb T^d_\theta)$ be self-adjoint and $f:\mathbb R\to\mathbb R$ be smooth when restricted to the spectrum of $x$. For every $m\in\mathbb Z_+$, we recursively define multilinear mappings $\mathbf T^{x,m}_{f}:\mathcal{X}^{\times m}\to C^\infty(\mathbb T^d_\theta)$ by firstly setting
$$\mathbf T^{x,m}_{f}(b_1,\ldots,b_m):=T^{x}_{f^{[m]}}(b_1,\ldots,b_m),$$
for all $b_1,\ldots,b_m\in C^\infty(\mathbb{T}^d_\theta)\subseteq\mathcal{X}$, secondly setting
\begin{align}
&\nonumber\mathbf T^{x,m}_{f}(\mathbf B_1,\ldots,\mathbf B_{k-1},\mathbf{B}_k\mathbf{D}_i,b_{k+1},\ldots,b_m)\\
&\qquad\qquad\qquad\qquad:=\nonumber\mathbf T^{x,m+1}_{f}(\mathbf B_1,\ldots,\mathbf B_{k},D_ix,b_{k+1},\ldots,b_m)\\
&\nonumber\qquad\qquad\qquad\qquad\quad\,+\mathbf T^{x,m}_{f}(\mathbf B_1,\ldots,\mathbf{B}_k,D_ib_{k+1},b_{k+2},\ldots,b_m)\\
&\qquad\qquad\qquad\qquad\quad\,+\mathbf T^{x,m}_{f}(\mathbf B_1,\ldots,\mathbf{B}_k,b_{k+1}\mathbf D_i,b_{k+2},\ldots,b_m),\label{eq:def bT 2}
\end{align}
for all $\mathbf B_1,\ldots,\mathbf B_m\in\mathcal{X}$ and $k<m$,
and lastly setting
\begin{align}\label{eq:def bT 3}
\mathbf T^{x,m}_{f}(\mathbf B_1,\ldots,\mathbf B_{m-1},\mathbf B_m\mathbf{D}_i):=\mathbf T^{x,m+1}_{f}(\mathbf B_1,\ldots,\mathbf B_m,D_ix).
\end{align}
\end{defi}
Well-definedness of $\mathbf T^{x,m}_{f}$ is shown in Lemma \ref{lem:T well defined}. 
\begin{ex}\label{ex:bT}
As a simple example of Definition \ref{def:bT} we have
\begin{align}\label{eq:first example bT}
\mathbf T^{x,3}_{f}(a,b\mathbf D_i)=T^{x}_{f^{[4]}}(a,b,D_ix).
\end{align}
A bit more work, but still easy is
\begin{align*}
\mathbf T^{x,3}_{f}(a\mathbf D_i,b,c)=&T^{x}_{f^{[4]}}(a,D_ix,b,c)+T^{x}_{f^{[3]}}(a,D_ib,c)+\mathbf T^{x,3}_{f}(a,b\mathbf D_i,c)\\
=&T^{x}_{f^{[4]}}(a,D_ix,b,c)+T^{x}_{f^{[3]}}(a,D_ib,c)+T^{x}_{f^{[4]}}(a,b,D_ix,c)\\
&+T^{x}_{f^{[3]}}(a,b,D_ic)+T^{x}_{f^{[4]}}(a,b,c,D_ix).
\end{align*}
An only slightly more involved expression like
$$\mathbf T^{x,3}_{f}(a\mathbf D_i\mathbf D_j,b\mathbf D_k,c)$$
already produces 145 terms, which can be straightforwardly obtained if one has enough time (or a computer at hand).
\end{ex}
 The above example illustrates that, however complicated $\mathbf B_1,\ldots,\mathbf B_m\in\mathcal X$ might be, $\mathbf T^{x,m}_{f}(\mathbf B_1,\ldots,\mathbf B_{m})$ can always be written as a sum of multiple operator integrals with arguments in $C^\infty(\mathbb T^d_\theta)$.
It also illustrates that, morally, we have 
\begin{align}\label{eq:moral identity}
\mathbf T^{x,m}_{f}(b_1\mathbf D^{\alpha_1},\ldots,b_m\mathbf D^{\alpha_m})=``\,T^{\lambda_l(x)}_{f^{[m]}}(\lambda_l(b_1)D^{\alpha_1},\ldots,\lambda_l(b_m)D^{\alpha_m})(1)\,",
\end{align}
in the sense that, if we would take \eqref{eq:MOI comm 3} at face value, we would have (using $D_i1=0$)
\begin{align*}
``T^{\lambda_l(x)}_{f^{[3]}}(\lambda_l(a),\lambda_l(b)D_i)(1)"&=T^{\lambda_l(x)}_{f^{[4]}}(\lambda_l(a),\lambda_l(b),[D_i,\lambda_l(x)])(1)\\
&=T^{\lambda_l(x)}_{f^{[4]}}(\lambda_l(a),\lambda_l(b),\lambda_l(D_ix))(1)\\
&=T^{x}_{f^{[4]}}(a,b,D_ix),
\end{align*}
which mimics \eqref{eq:first example bT}, and similarly for the other defining properties of $\mathbf T^{x,m}_f$.
However, the unbounded arguments of the multiple operator integrals between quotes warrant some caution. The moral identity \eqref{eq:moral identity} is made rigorous by Corollary \ref{cor:S and bT}, which forms a crucial step towards our main theorem.
 
The final ingredients for our main result are the elements that we use as inputs of the mappings $\mathbf T^{x,m}_f$.
Let $P$ be as in \eqref{eq:P}.
For every $m\in\mathbb N$, every subset $\mathscr{A}\subseteq\{1,\ldots,m\}$, and every function $\iota:\mathscr{A}\to\{1,\ldots,d\}$, we define $\mathbf W^{\mathscr A,\iota}_1,\ldots,\mathbf W^{\mathscr A,\iota}_m\in\mathcal X$ by
\begin{align}\label{eq:bW}
\mathbf W^{\mathscr{A},\iota}_j=
\begin{cases}
{\bf A}_{\iota(j)}\quad&(j\in\mathscr{A});\\
{\bf P},&(j\notin\mathscr{A}),
\end{cases}
\end{align}
where (for all $i\in\{1,\ldots,d\}$)
\begin{align}\label{eq:bf Ai and bf P}
{\bf A}_i:=2x\bD_i+a_i,\quad {\bf P}:=x\sum_{i=1}^d \bD_i^2+\sum_{i=1}^da_i\bD_i+a\quad\in\quad\mathcal X.
\end{align}

\subsection{Main result}
Our main result is formulated as follows.

\begin{thm}\label{thm: recursive formula}
Let $d\in\mathbb N_{\geq 2}$, $k\in 2\mathbb Z_+,$ and let $P$ be a self-adjoint operator acting in $L_2(\mathbb T^d_\theta)$ of the form \eqref{eq:P} for positive invertible $x$. The $k^{\text{th}}$ order local invariant of $P$ occurring in the asymptotic expansion \eqref{eq:asymptotic expansion} takes the form
\begin{align}\label{eq:main thm}
I_k(P)=(-1)^{\frac{k}{2}}\pi^{\frac{d}{2}}\sum_{\frac{k}{2}\leq m\leq k}\,\sum_{\substack{\mathscr A\subseteq\{1,\ldots,m\}\\|\mathscr A|=2m-k}}\,\sum_{\iota:\mathscr{A}\to\{1,\ldots,d\}}
c_d^{(\iota)}\mathbf T^{x,m}_{F_{k,d}}(\mathbf W^{\mathscr{A},\iota}_1,\ldots,\mathbf W^{\mathscr{A},\iota}_m),
\end{align}
where $F_{k,d}$ is any $\frac{k}{2}^{\text{th}}$ order primitive of $\alpha\mapsto\alpha^{-\frac{d}{2}}$ and
$$c_d^{(\iota)}:=\frac{1}{{\rm vol}(\mathbb S^{d-1})}\int_{\mathbb S^{d-1}}\prod_{j\in\mathscr A} u_{\iota(j)}\,du.$$
For non-self-adjoint $P$ the right-hand side of \eqref{eq:main thm} still exists, and we may take this as the extended definition of $I_k(P)$ (as it coincides with Definition \ref{ik notation}).
\end{thm}

\begin{rem}\label{rem:Fkd}
An explicit $\frac{k}{2}^{\text{th}}$ order primitive of $\alpha\mapsto\alpha^{-\frac{d}{2}}$ is given by
$$F_{k,d}(\alpha):=\begin{cases}
(-1)^{\frac{k}{2}}\frac{\Gamma(\frac{d}{2}-\frac{k}{2})}{\Gamma(\frac{d}{2})}\alpha^{\frac{k-d}{2}}\quad&\text{if $d$ is odd or $k<d$;}\\
(-1)^{\frac{d}{2}-1}\frac{1}{(\frac{d}{2}-1)!(\frac{k}{2}-\frac{d}{2})!}\alpha^{\frac{k-d}{2}}\log(\alpha)\quad&\text{if $d$ is even and $k\geq d$.}
\end{cases}$$
In particular, we have $F_{2,2}=\log$.
\end{rem}

\begin{rem}\label{rem:c_d}
The constants $c_d^{(\iota)}$ are rational, invariant under permutations on the domain and range of $\iota$, and, lastly, easy to compute. Writing $n_j:=|\iota^{-1}(\{j\})|$, we have (cf. \cite{Folland})
$$c_d^{(\iota)}=
\begin{cases}
\frac{(d-2)!!\prod_{j=1}^d(n_j-1)!!}{(|n|_1+d-2)!!}\quad&\text{if $n_1,\ldots n_d$ are even;}\\
0\quad&\text{otherwise.}
\end{cases}$$
Here we use the usual convention $(-1)!!=1$.
\end{rem}

To illustrate what \eqref{eq:main thm} means in practice, we note that for $k=0,2,4$ it states that (cf. Section \ref{sct:k=2})
\begin{align*}
\pi^{-\frac{d}{2}}I_0(P)=&\mathbf T^{x,0}_{F_{0,d}}();\\
-\pi^{-\frac{d}{2}}I_2(P)=&\mathbf T^{x,1}_{F_{2,d}}(\mathbf P)+\sum_{i=1}^d\frac1d\mathbf T^{x,2}_{F_{2,d}}(\mathbf A_i,\mathbf A_i);\\
\pi^{-\frac{d}{2}}I_4(P)=&\mathbf T^{x,2}_{F_{4,d}}(\mathbf P,\mathbf P)+\sum_{i=1}^d\frac1d\Big(\mathbf T^{x,3}_{F_{4,d}}(\mathbf P,\mathbf A_i,\mathbf A_i)+\mathbf T^{x,3}_{F_{4,d}}(\mathbf A_i,\mathbf P,\mathbf A_i)\\
&+\mathbf T^{x,3}_{F_{4,d}}(\mathbf A_i,\mathbf A_i,\mathbf P)\Big)+\sum_{i,j,k,l=1}^dc_d^{(i,j,k,l)}\mathbf T^{x,4}_{F_{4,d}}(\mathbf A_i,\mathbf A_j,\mathbf A_k,\mathbf A_l),
\end{align*}
but the real beauty of \eqref{eq:main thm} is that this compact expression holds for any $k$.

\subsection{Consequences of our main result}

Straightforward corollaries of Theorem \ref{thm: recursive formula} are obtained by fixing $k$ and expanding the recursive definition of $\mathbf T^{x,m}_f$ (as in Remark \ref{rem:algebraic MOI}) into explicit sums of multiple operator integrals with arguments in the non-commutative torus. The resulting formula for $I_0$ is nothing new, namely
$$\pi^{-\frac{d}{2}}I_0(P)=T^x_{F_{0,d}^{[0]}}()=F_{0,d}(x)=x^{-\frac{d}{2}}.$$ However, the resulting formula for the second order local invariant $I_2$, which is sometimes called the scalar curvature, is already of note. We obtain the following explicit expression, which for $d=2,4$ can be used to recover the results of \cite{Connes-Moscovici,Fa,FaKh1,FaKh} (more on this later).

\begin{cor}\label{cor:main thm k=2} 
For any dimension $d\in\mathbb N_{\geq2}$, and $P$ acting in $L_2(\mathbb T^d_\theta)$ of the form \eqref{eq:P} for positive invertible $x$, the second order local invariant of $P$ is computed by
\begin{align*}
-\pi^{-\frac{d}{2}}I_2(P)=&\sum_{i=1}^d\Big(
2T^x_{F_{2,d}^{[3]}}(x,D_ix,D_ix)
+T^x_{F_{2,d}^{[2]}}(a_i,D_ix)\Big)
+T^x_{F_{2,d}^{[2]}}(x,\Delta x)\\
&+T^x_{F_{2,d}^{[1]}}(a)+
\sum_{i=1}^d\frac{1}{d}\Big(
4T^x_{F_{2,d}^{[4]}}(x,D_ix,x,D_ix)
+4T^x_{F_{2,d}^{[3]}}(x,D_ix,D_ix)\\
&+8T^x_{F_{2,d}^{[4]}}(x,x,D_ix,D_ix)
+2T^x_{F_{2,d}^{[3]}}(x,D_ix,a_i)
+2T^x_{F_{2,d}^{[2]}}(x,D_ia_i)\\
&+2T^x_{F_{2,d}^{[3]}}(x,a_i,D_ix)
+2T^x_{F_{2,d}^{[3]}}(a_i,x,D_ix)
+T^x_{F_{2,d}^{[2]}}(a_i,a_i)
\Big)\\
&+\frac{4}{d}T^x_{F_{2,d}^{[3]}}(x,x,\Delta x).
\end{align*}
\end{cor}
Deriving the above formula from our main theorem is quite straightforward; Section \ref{sct:k=2} contains an explicit proof for convenience of the reader.

In fact, the same can be done for any order $k$ in a simple manner.
\begin{cor}\label{main thm k=k}
For any $d\in\mathbb N_{\geq2}$, $k\in 2\mathbb Z_+$, and $P$ acting in $L_2(\mathbb T^d_\theta)$ of the form \eqref{eq:P} for positive invertible $x$, an expression for the $k^{\text{th}}$ order local invariant $I_k(P)$ can be computed by the accompanying python program (cf. Appendix \ref{appendix A}). This expression consists of a finite amount of terms of the form
$$c\pi^{\frac{d}{2}}T^{x}_{F_{k,d}^{[m]}}(D^{\alpha_1}b_1,\ldots,D^{\alpha_m}b_m),$$
where $m\in\mathbb N$, $c\in\mathbb Q$, $\alpha_j\in\mathbb Z_+^d$ and $b_j\in\{x,a_1,\ldots,a_d,a\}$. E.g., $I_4$ has 1046 terms and $I_6$ has 140845 terms in Einstein notation (i.e., not counting sums over indices $i_j=1,\ldots,d$). 
\end{cor}

For specific $P$ and $k$, the above expressions can yield remarkably elegant results. As an example we shall focus on $k=2$ and the case $P=\lambda_l(x^{1/2})\Delta\lambda_l(x^{1/2})$, which corresponds to the Laplacian `on functions' (see \cite{Connes-Moscovici,FaKh1} for terminology) of the conformally deformed non-commutative 2-torus. In this case (and in fact in the analogous case for every $d\geq2$) we obtain an expression for $I_2(P)$ that is arguably neater than the expressions one finds in the literature (cf. \cite{Connes-Moscovici}).

\begin{thm}\label{cor:main thm k=d=2} Let $d\geq2$ and consider $P=\lambda_l(x^{\frac12})\Delta\lambda_l(x^{\frac12})$ acting in $L_2(\mathbb T^d_\theta)$ for a positive invertible $x\in C^\infty(\mathbb T^d_\theta)$. The second order local invariant of $P$ is given by
$$-\pi^{-\frac{d}{2}}I_2(P)=T^x_{\Phi}(\Delta x)+\sum_{i=1}^dT^x_{\Psi}(D_ix,D_ix),$$
where $\Phi$ and $\Psi$ are expressed in terms of divided differences as
$$\Phi(\alpha_0,\alpha_1)=\frac{2(\alpha_0\alpha_1)^{\frac12}}{d}\cdot\frac{\alpha_0F_{2,d}^{[2]}(\alpha_0,\alpha_0,\alpha_1)-\alpha_1F_{2,d}^{[2]}(\alpha_0,\alpha_1,\alpha_1)}{\alpha_1-\alpha_0},$$
$$\Psi(\alpha_0,\alpha_1,\alpha_2)=-\frac4d
\frac{(\alpha_0\alpha_2)^{\frac12}}{\alpha_1^{2+\frac{d}{2}}}
g^{[3]}\Big(\frac{\alpha_0}{\alpha_1},\frac{\alpha_0}{\alpha_1},\frac{\alpha_2}{\alpha_1},\frac{\alpha_2}{\alpha_1}\Big),\quad  g(\alpha)=F_{2,d}(\alpha)+F_{2,d}^{[1]}(1,\alpha),$$
for all $\alpha_0,\alpha_1,\alpha_2,\alpha>0$.
\end{thm}
Multiple operator integrals can also serve as a stepping stone towards the modular functional calculus ubiquitous in the literature since \cite{Cohen-Connes,Connes-Moscovici,Connes-Tretkoff}. Indeed, from the above formula one can derive the most basic main results of \cite{Connes-Moscovici,FaKh1} as a corollary, namely the functions $K_0(s)$ and $H_0(s,t)$ from \cite{Connes-Moscovici}. This derivation is done in Section \ref{sct:Connes-Moscovici}. 

In fact, as our main result holds for arbitrary $d$ and $k$, many more `modular formulas' are now within easy reach. As a quick example, if $k=2$ and $d\geq3$ is arbitrary, then Theorem \ref{thm:K0 arbitrary d} (which can be derived from Theorem \ref{cor:main thm k=d=2} or directly from our main theorem) shows how the function
$$K_0^d(s)=\frac{2}{d}\cdot \frac{-1-e^{(1-\frac{d}{2})s}+\frac{e^{\left(1-d/2\right)s}-1}{1-d/2}\coth\left(\frac{s}{2}\right)}{s\sinh\left(\frac{s}{2}\right)}$$
replaces the function $K_0(s)$ of \cite[eq. (2)]{Connes-Moscovici} when passing to arbitrary dimension. Moreover, one immediately recovers the function $K_0(s)$ of \cite{Connes-Moscovici} by taking $d\to2$ in the above formula. 

The recursive formula of Theorem \ref{thm: recursive formula} bears similarity to some of the formulas in \cite{IM1,IM2,IM3}. Indeed, we show how to recover a key result from \cite{IM1} from our main theorem in Theorem \ref{thm:IM}, thus finally bridging the gap between the commutative and non-commutative approaches.


\section{Groundwork}\label{prelim section}

\begin{lem}\label{lem:T well defined}
The map $\mathbf T:\mathcal{X}^{\times n}\to C^\infty(\mathbb T^d_\theta)$ of Definition \ref{def:bT} is well defined.
\end{lem}
\begin{proof}
It suffices to show that the expression defining
\begin{align}\label{eq:weldefinedness LHS}\mathbf T^{x,n}_{f}(\mathbf B_1,\ldots,\mathbf B_{k-1},\mathbf{B}_k\mathbf{D}_i\mathbf{D}_j,b_{k+1},\ldots,b_n)
\end{align}
equals the expression defining
\begin{align}\label{eq:weldefinedness RHS}
\mathbf T^{x,n}_{f}(\mathbf B_1,\ldots,\mathbf B_{k-1},\mathbf{B}_k\mathbf{D}_j\mathbf{D}_i,b_{k+1},\ldots,b_n).
\end{align}
By induction, one can show that \eqref{eq:weldefinedness LHS} is equal to a long expression involving $D_i$ and $D_j$ occuring in the arguments after $\mathbf B_1,\ldots,\mathbf B_k$ in one of the following forms 
$$\ldots,D_ix,\ldots,D_jx,\ldots,\qquad\ldots,D_jx,\ldots,D_ix,\ldots,$$
$$\ldots,D_ib_l,\ldots,D_jx,\ldots,\qquad\ldots,D_jb_l,\ldots,D_ix,\ldots,$$
$$\ldots,D_ix,\ldots,D_jb_m,\ldots,\qquad\ldots,D_jx,\ldots,D_ib_m,\ldots,$$
$$\ldots,D_ib_l,\ldots,D_jb_m,\ldots,\qquad\ldots,D_jb_l,\ldots,D_ib_m,\ldots,$$
$$\ldots,D_iD_jx,\ldots,\qquad\ldots,D_iD_jb_l,\ldots,$$
where the dots signify the list of other arguments $b_{k+1},\ldots,b_n$, cut up at arbitrary places. One sees that the first 8 instances are in bijection with one another after swapping $i$ and $j$. The last 2 instances are invariant under swapping $i$ and $j$ because $D_iD_j=D_jD_i$. Hence \eqref{eq:weldefinedness LHS} is equal to \eqref{eq:weldefinedness RHS}.
\end{proof}

\subsection{The results of our companion paper}\label{defnot section}
In our companion paper \cite{SZ-local} the existence of the asymptotic expansion was proven for the present general class of operators $P$ (in fact, for an even more general assumption on the scalar symbol), and a formula was given for $I_k(P)$. However, this formula was not explicit.

As in \cite{SZ-local}, this formula is stated as a definition of $I_k(P)$ for all $P$ of the form \eqref{eq:P} with $x$ self-adjoint and invertible. If $P$ is in addition self-adjoint, this definition of $I_k(P)$ coincides with the definition in the introduction (see Theorem \ref{companion thm} below).
\begin{defi}
For $s\in\mathbb{R}^d,$ we define
\begin{align*}
V(s)&:=\sum_{i=1}^ds_iA_i,\\
A_i&:=2\lambda_l(x)D_i+\lambda_l(a_i),\quad 1\leq i\leq d,
\end{align*}
as linear operators $C^\infty(\mathbb T^d_\theta)\to L_2(\mathbb T^d_\theta)$ acting (densely) in $L_2(\mathbb T^d_\theta)$.
\end{defi}
As $x$ is positive and invertible, $x|s|^2+z\in C^{\infty}(\mathbb{T}^d_{\theta})$ is invertible in $L_{\infty}(\mathbb{T}^d_{\theta})$ for every $z\in\mathbb{C}\backslash\mathbb{R}_-.$ As the $C^{\infty}(\mathbb{T}^d_{\theta})$ is stable under holomorphic functional calculus, we have $(x|s|^2+z)^{-1}\in C^{\infty}(\mathbb{T}^d_{\theta})$.

%
\begin{defi}\label{fourth resolvent nota} Let $\mathscr{A}\subseteq\mathbb{N}.$ For every $z\in\mathbb{C}\backslash\mathbb{R}_-$ and every $s\in\mathbb{R}^d,$ set
$f_0^{\mathscr{A}}(s,z):=1$ and
$$f_m^{\mathscr{A}}(s,z):=W^\mathscr{A}_j(s)\Big(\frac{1}{x|s|^2+z}f_{m-1}^{\mathscr{A}}(s,z)\Big),\quad m\geq1,$$
where (cf. \eqref{eq:bW})
\begin{align}\label{eq:W notation}
W^\mathscr{A}_j(s):=
\begin{cases}
V(s)\quad &(j\in\mathscr{A});\\
P\quad &(j\notin\mathscr{A}).
\end{cases}
\end{align}
\end{defi}

\begin{defi}\label{fifth resolvent nota}
For every $z\in\mathbb{C}\backslash\mathbb{R}_-$ every $s\in\mathbb{R}^d$ and every $k\in\mathbb{Z}_+$ we set
$${\rm corr}_k(s,z):=(x|s|^2+z)^{-1}\sum_{\frac{k}{2}\leq m\leq k}(-1)^m\sum_{\substack{\mathscr{A}\subseteq\{1,\cdots,m\}\\ |\mathscr{A}|=2m-k}}f_m^{\mathscr{A}}(s,z).$$
\end{defi}

\begin{defi}\label{first exponential notation} For every $s\in\mathbb{R}^d$ and every $k\in\mathbb{Z}_+$ we set
$${\rm Corr}_k(s):=\frac1{2\pi}\int_{-\infty}^\infty{\rm corr}_k(s,i\lambda)e^{i\lambda}\,d\lambda.$$
Here and throughout this paper, $\int_{-\infty}^\infty:=\lim_{N\to\infty}\int_{-N}^N$. In the case above, the limit is with respect to the weak operator topology. The distinction between $\int_{-\infty}^\infty$ and the Lebesgue integral $\int_{\mathbb R}$ is only relevant in the case $k=0$.
\end{defi}
\begin{defi}\label{ik notation} For every $k\in\mathbb{Z}_+,$ we define
$$I_k(P):=\int_{\mathbb{R}^d}{\rm Corr}_k(s)\,ds,$$
as a weak integral in $L_\infty(\mathbb T^d_\theta)$.
\end{defi}

Theorem 1.2 in \cite{SZ-local} asserts the following.
\begin{thm}\label{companion thm} If $P$ is self-adjoint acting in $L_2(\mathbb T^d_\theta)$ of the form \eqref{eq:P} for $x$ positive and invertible, then \eqref{eq:asymptotic expansion} holds with $\{I_k(P)\}_{k\geq0}$ as in Notation \ref{ik notation}, and $I_k(P)\in C^\infty(\mathbb T^d_\theta)$ for all $k\geq0$.
\end{thm}

In the next three sections we will rewrite the above definition into a computable formula for $I_k(P)$, and thus prove our main theorem.

\section{Recursion at the level of symbols}\label{sct:initial}

Recall that  
$\mathcal{X}:= {\rm span}\big\{b\mathbf{D}^\alpha~\big|~b\in C^\infty(\mathbb{T}^d_\theta),\alpha\in\mathbb Z_+^d\big\}$, a free $C^\infty(\mathbb T^d_\theta)$-module. Similarly, let $X:=\big\{\lambda_l(b)D^\alpha:C^\infty(\mathbb T^d_\theta)\to C^\infty(\mathbb T^d_\theta)~\big|~b\in C^\infty(\mathbb T^d_\theta),\alpha\in\mathbb Z_+^d\big\}$ be the $C^\infty(\mathbb{T}^d_\theta)$-module generated by the operators $D^\alpha=D_1^{\alpha_1}\cdots D_d^{\alpha_d}$, $\alpha\in\mathbb Z_+^d$, seen here simply as linear functions from $C^\infty(\mathbb{T}^d_\theta)$ to $C^\infty(\mathbb{T}^d_\theta)$. 
We define a $C^\infty(\mathbb T^d_\theta)$-module homomorphism $\pi:\mathcal{X}\to X$ by linear extension of 
\begin{align}\label{eq:def pi}
\pi(b\bD^\alpha):=\lambda_l(b)D^\alpha.
\end{align}
Fix a positive invertible $x\in C^\infty(\mathbb T^d_\theta)$. We define multilinear mappings $S_{s,z}^{m}:\mathcal{X}^{\times m}\to C^\infty(\mathbb{T}^d_\theta)$ for every $m\in\mathbb{N}$, $s\in\mathbb R^d\setminus\{0\}$ and $z\in \mathbb C\setminus\mathbb R_-$ by
\begin{align}\label{eq:def S^m_sz}
&S_{s,z}^{m}(\mathbf{B}_1,\ldots,\mathbf{B}_m):=\\
(-1)^m&|s|^{2m}\frac{1}{x|s|^2+z}\pi(\mathbf{B}_1)\Big(\frac{1}{x|s|^2+z}\cdots \pi(\mathbf{B}_m)\Big(\frac{1}{x|s|^2+z}\Big)\cdots\Big),\nonumber
\end{align}
for all $\mathbf{B}_1,\ldots,\mathbf{B}_m\in\mathcal{X}$. The above expression is well-defined because $(x|s|^2+z)^{-1}\in C^\infty(\mathbb T^d_\theta)$ and elements of $X$ preserve $C^\infty(\mathbb T^d_\theta)$. In the following subsection we show that $S_{s,z}^m$ satisfies the same recursive properties as $\mathbf T^{x,m}_{f}$ does by Definition \ref{def:bT} (i.e., \eqref{eq:def bT 2} and \eqref{eq:def bT 3}). Relating $S_{s,z}^{n}(b_1,\ldots,b_n)$ to $\mathbf T^{x,n}_{f}(b_1,\ldots,b_n)$ for $b_i\in C^\infty(\mathbb T^d_\theta)$ (i.e., relating the two base cases of the respective recursions) involves some heavy analysis, and is done in Section \ref{sct:Integral formula S and T}.


\subsection{Recursive formula for $S_{s,z}$}

The following two lemmas show how expressions of the form $S_{s,z}^{m}(\mathbf{B}_1,\ldots,\mathbf{B}_m)$ (where $\mathbf{B}_i\in\mathcal{X}$) can be rewritten in terms of expressions of the form $S_{s,z}^{n}(b_1,\ldots,b_n)$, where $b_i\in C^\infty(\mathbb{T}^d_\theta)$ and $n\geq m$.

\begin{lem}\label{zeroth workhorse lemma}
Let $k,m\in\mathbb N$, $k<m$, $\mathbf B_1,\ldots,\mathbf B_k\in\mathcal{X}$, $b_{k+1},\ldots,b_m\in C^\infty(\mathbb T_\theta^d)\subseteq\mathcal{X}$, and $i\in\{1,\ldots,d\}$. For all $s\in\mathbb R^d\setminus\{0\}$ and $z\in \mathbb C\setminus\mathbb R_-$ we have
\begin{align*}
S_{s,z}^m(\mathbf B_1,\ldots,\mathbf B_{k-1},\mathbf{B}_k\mathbf{D}_i,b_{k+1},\ldots,b_m)&=S_{s,z}^{m+1}(\mathbf B_1,\ldots,\mathbf B_{k},D_ix,b_{k+1},\ldots,b_m)\\
\nonumber&\,\,+S_{s,z}^m(\mathbf B_1,\ldots,\mathbf{B}_k,D_ib_{k+1},b_{k+2},\ldots,b_m)\\
&\,\,+S_{s,z}^m(\mathbf B_1,\ldots,\mathbf{B}_k,b_{k+1}\mathbf D_i,b_{k+2},\ldots,b_m).
\end{align*}
\end{lem}
\begin{proof} 
We first note that from $D_i(u\cdot u^{-1})=0$ it follows that
$$D_i(u^{-1})=-u^{-1}\cdot D_iu\cdot u^{-1}.$$
Thusly, we obtain
$$D_i\Big(\frac{1}{x|s|^2+z}\Big)=-|s|^2\frac{1}{x|s|^2+z}(D_ix)\frac{1}{x|s|^2+z}.$$
By the latter equality and the Leibniz rule we obtain, for any $k\in \mathbb Z$,
\begin{align*}
\pi(\bD_i)&\Big(\frac{1}{x|s|^2+z} b_{k+1}\frac{1}{x|s|^2+z}\cdots b_m\frac{1}{x|s|^2+z}\Big)\\
&=D_i\Big(\frac{1}{x|s|^2+z} b_{k+1}\frac{1}{x|s|^2+z}\cdots b_m\frac{1}{x|s|^2+z}\Big)\\
&=-|s|^2\frac{1}{x|s|^2+z}D_ix\frac{1}{x|s|^2+z}b_{k+1}\frac{1}{x|s|^2+z}\cdots b_n\frac{1}{x|s|^2+z}\\
&\quad+\frac{1}{x|s|^2+z}D_ib_{k+1}\frac{1}{x|s|^2+z}b_{k+2}\frac{1}{x|s|^2+z}\cdots b_n\frac{1}{x|s|^2+z}\\
&\quad+\frac{1}{x|s|^2+z}b_{k+1}D_i\Big(\frac{1}{x|s|^2+z}b_{k+2}\frac{1}{x|s|^2+z}\cdots b_n\frac{1}{x|s|^2+z}\Big)\\
&=-|s|^2\frac{1}{x|s|^2+z}D_ix\frac{1}{x|s|^2+z}b_{k+1}\frac{1}{x|s|^2+z}\cdots b_n\frac{1}{x|s|^2+z}\\
&\quad+\frac{1}{x|s|^2+z}D_ib_{k+1}\frac{1}{x|s|^2+z}b_{k+2}\frac{1}{x|s|^2+z}\cdots b_n\frac{1}{x|s|^2+z}\\
&\quad+\frac{1}{x|s|^2+z}\pi(b_{k+1}\bD_i)\Big(\frac{1}{x|s|^2+z}b_{k+2}\frac{1}{x|s|^2+z}\cdots b_n\frac{1}{x|s|^2+z}\Big).
\end{align*}
After multiplying both sides by $(-1)^{m-k}|s|^{2(m-k)}$, the above equality becomes
\begin{align}\label{eq:S^{m-k}}
\pi(\mathbf D_i)\Big(S^{m-k}_{s,z}(b_{k+1},\ldots,b_m)\Big)=&S^{m-k+1}_{s,z}(D_ix,b_{k+1},\ldots,b_m)\\
&+S^{m-k}_{s,z}(D_ib_{k+1},b_{k+2},\ldots,b_m)\nonumber\\
&+S^{m-k}_{s,z}(b_{k+1}\mathbf D_i,b_{k+2},\ldots,b_m).\nonumber
\end{align}
As
\begin{align*}
&S_{s,z}^{m}(\mathbf B_1,\ldots,\mathbf B_k,\mathbf B_{k+1},\ldots,\mathbf B_m)\\
&=(-1)^{k}|s|^{2k}\frac{1}{x|s|^2+z}\pi(\mathbf B_1)\Big(\cdots \frac{1}{x|s|^2+z} \pi(\mathbf B_k)\Big(S_{s,z}^{m-k}(\mathbf B_{k+1},\ldots,\mathbf B_m)\Big)\cdots \Big),
\end{align*}
the lemma follows from \eqref{eq:S^{m-k}}.

\end{proof}

\begin{lem}\label{first workhorse lemma} Let $m\in\mathbb{N}$, $\mathbf B_1,\ldots,\mathbf B_m\in\mathcal{X}$ and $i\in\{1,\ldots,d\}$. 
For all $s\in\mathbb R^d\setminus\{0\}$ and $z\in \mathbb C\setminus\mathbb R_-$ we have
$$S_{s,z}^{m}(\mathbf B_1,\ldots,\mathbf B_{m-1},\mathbf B_m\mathbf{D}_i):=S_{s,z}^{m+1}(\mathbf B_1,\ldots,\mathbf B_m,D_ix).$$
\end{lem}
\begin{proof} This is an easier version of the proof of Lemma \ref{zeroth workhorse lemma}.
\end{proof}

\section{Analytical results on multiple operator integrals}
\label{sct:Integral formula S and T}

The purpose of this section is to prove the following theorem.
\begin{thm}\label{ik base of recursion} Let $x\in C^\infty(\mathbb T^d_\theta)$ be positive and invertible. For every $n\in\mathbb N$, $b_1,\ldots,b_n\in C^{\infty}(\mathbb{T}^d_\theta)$, and every $k\in 2\mathbb{Z}_+$ such that $2n-k\geq0$, we have
$$\frac{1}{2\pi}\int_{\mathbb{R}^d}\Big(\int_{-\infty}^\infty|s|^{-k}S^n_{s,i\lambda}(b_1,\ldots,b_n)e^{i\lambda}d\lambda\Big)ds=(-1)^{\frac{k}{2}}\pi^{\frac{d}{2}}\cdot T^{x}_{F_{k,d}^{[n]}}(b_1,\ldots,b_n),$$
where $F_{k,d}$ is any $\frac{k}{2}^{\text{th}}$ primitive of $\alpha\mapsto\alpha^{-\frac{d}{2}}$.
\end{thm}

For any open interval $I\subseteq \mathbb R$ and any $n\in\mathbb N$ we will use the space
$$\dot W^{n,2}(I):=\{f\in \mathcal S'(I):~f^{(n)}\in L_2(I)\}$$
(where $\mathcal S'(I)$ denotes the tempered distributions on $I$) with associated seminorm
$$\|f\|_{\dot W^{n,2}(I)}:=\|f^{(n)}\|_{L_2(I)}\,.$$
By slight abuse of notation, we denote by $(\dot W^{n,2}\cap \dot W^{n+1,2})(I)$ the space of equivalence classes of functions in $\dot W^{n,2}(I)\cap \dot W^{n+1,2}(I)$ modulo polynomials of degree at most $n-1$. We omit the notation for `the equivalence class of'. We equip $(\dot W^{n,2}\cap \dot W^{n+1,2})(I)$ with the norm
$$\|f\|_{(\dot W^{n,2}\cap \dot W^{n+1,2})(I)}:=\|f^{(n)}\|_{L_2(I)}+\|f^{(n+1)}\|_{L_2(I)}\,.$$
This space $(\dot W^{n,2}\cap \dot W^{n+1,2})(I)$ is a Banach space, as can be shown by standard techniques. Note also that any $f\in\dot{W}^{n,2}(I)$ is a continuous function, because $f^{(n)}$ is locally integrable. Hence, any representative of a class in $(\dot W^{n,2}\cap \dot W^{n+1,2})(I)$ is a continuous function.

\begin{lem}\label{lem:MOI bound Banach}
Let $x\in L_\infty(\mathbb T^d_\theta)$ be self-adjoint and let $I$ be an open interval containing ${\rm spec}(x).$ Let $f$ be a Schwartz function on $\mathbb R$. For all $b_1,\ldots,b_n\in L_\infty(\mathbb{T}^d_\theta)$ we have
$$\|T^{x}_{f^{[n]}}(b_1,\cdots,b_n)\|_{\infty}\leq c_{n,x,I}\|f\|_{(\dot{W}^{n,2}\cap \dot{W}^{n+1,2})(I)}\prod_{l=1}^n\|b_l\|_{\infty}.$$
\end{lem}
\begin{proof} Let $J=[\inf{\rm spec}(x),\sup{\rm spec}(x)]\subseteq I.$ Let $\phi$ be a smooth function supported in $I$ such that $\phi$ equals $1$ on $J.$ 
We have
$$f^{[n]}(\alpha_0,\cdots,\alpha_n)=\int_{S^n}f^{(n)}\bigg(\sum_{j=0}^n\lambda_j\alpha_j\bigg)d\lambda,$$
where the integration is taken with respect to the standard measure on the simplex $S^n=\{\lambda\in\mathbb R_{\geq0}^{n+1}:~\sum_{j=0}^n\lambda_j=1\}.$ If $\alpha_0,\cdots,\alpha_n\in{\rm spec}(x),$ then
$\sum_{j=0}^n\lambda_j\alpha_j\in J.$ Therefore, denoting the Fourier transform of the Schwartz function $f^{(n)}\phi$ by $\widehat{(f^{(n)}\phi)}$, we have
\begin{align*}
f^{[n]}(\alpha_0,\cdots,\alpha_n)&=\int_{S^n}(f^{(n)}\phi)\bigg(\sum_{j=0}^n\lambda_j\alpha_j\bigg)d\lambda\\
&=\int_{S^n}\int_\mathbb{R}\widehat{(f^{(n)}\phi)}(t)e^{it\lambda_0\alpha_0}\cdots e^{it\lambda_n\alpha_n}dt\,d\lambda,
\end{align*}
whenever $\alpha_0,\cdots,\alpha_n\in{\rm spec}(x).$ Thus,
$$\|T^{x}_{f^{[n]}}(b_1,\cdots,b_n)\|_{\infty}\leq\frac{1}{n!}\|\widehat{(f^{(n)}\phi)}\|_{L_1(\mathbb{R})}\prod_{l=1}^n\|b_l\|_{\infty}.$$
Note that
$$\|\widehat{(f^{(n)}\phi)}\|_{L_1(\mathbb{R})}\leq\sqrt{2}\big(\|f^{(n)}\phi\|_{L_2(I)}+\|(f^{(n)}\phi)'\|_{L_2(I)}\big).$$
By the Leibniz rule, we deduce
$$\|\widehat{(f^{(n)}\phi)}\|_{L_1(\mathbb{R})}\leq \|f^{(n)}\|_{L_2(I)}(\|\phi\|_{\infty}+\|\phi'\|_{\infty})+\|f^{(n+1)}\|_{L_2(I)}\|\phi\|_{\infty}.$$
Since $\phi$ depends only on $x$ and $I,$ the assertion follows.
\end{proof}

\subsection{Integration over the symbol of a multiple operator integral}
In this subsection we prove the following general result. We let $(\sigma_tf)(\alpha):=f(\alpha/t)$ denote the dilation operator.

\begin{thm}\label{moi integration thm} Let $x\in L_\infty(\mathbb T^d_\theta)$ be positive and invertible. Let $f$ be a Schwartz function on $\mathbb{R}.$ Let $k\in 2\mathbb{Z}_+$ and $n\in\mathbb{N}$ be such that $2n\geq k.$ For all $b_1,\ldots,b_n\in L_\infty(\mathbb{T}^d_\theta)$ we have
$$\int_{\mathbb{R}^d}|s|^{-k}T^{x}_{(\sigma_{|s|^{-2}}f)^{[n]}}(b_1,\ldots,b_n)\,ds=\int_{\mathbb{R}^d}f^{(\frac{k}{2})}(|s|^2)ds\cdot T^{x}_{F_{k,d}^{[n]}}(b_1,\ldots,b_n),$$
where the left-hand side is a Bochner integral taking values in $L_\infty(\mathbb T^d_\theta)$, and $F_{k,d}$ is any $\frac{k}{2}^{\text{th}}$ primitive of $\alpha\mapsto\alpha^{-\frac{d}{2}}$ on $(0,\infty)$ (cf. Remark \ref{rem:Fkd}).
\end{thm}

\begin{lem}\label{Fk lemma} Let $k\in 2\mathbb{Z}_+.$ Let $f$ be a Schwartz function on $\mathbb{R}$ and let $\phi$ be a Schwartz function on $\mathbb{R}^d$ that equals 1 on a neighbourhood of $0.$ There exists a $\frac{k}{2}^\text{th}$ order primitive $F_{k,d}$ of $\alpha\mapsto\alpha^{-\frac{d}{2}}$ such that, for all $\alpha>0$,
$$\int_{\mathbb{R}^d}|s|^{-k}\Big(f(\alpha|s|^2)-\sum_{j=0}^{\frac{k}{2}-1}\frac{f^{(j)}(0)}{j!}(\alpha|s|^2)^{j}\cdot\phi(s)\Big)ds=\int_{\mathbb{R}^d}f^{(\frac{k}{2})}(|s|^2)ds\cdot F_{k,d}(\alpha).$$
\end{lem}
\begin{proof} The left-hand side integral converges because, as $s\to 0$, the expression between brackets is $\mathcal{O}(|s|^k)$. Similarly this integral converges after differentiating the integrand with respect to $\alpha$. Denoting the left-hand side by $G(\alpha)$, we have
$$G^{(\frac{k}{2})}(\alpha)=\int_{\mathbb{R}^d}f^{(\frac{k}{2})}(\alpha |s|^2)ds=\alpha^{-\frac{d}{2}}\int_{\mathbb{R}^d}f^{(\frac{k}{2})}(|s|^2)ds.$$
Integrating $\frac{k}{2}$ times, we complete the proof.
\end{proof}

We have the following simple but powerful proposition.

\begin{prop}\label{wiener moi fact} Let $n\in \mathbb N$ and let $x\in L_\infty(\mathbb T^d_\theta)$ be self-adjoint. Let $I$ be an open interval containing the spectrum of $x$. Let $s\mapsto h_s$ $(s\in\mathbb{R}^d)$ be a Bochner integrable mapping taking values in $(\dot{W}^{n,2}\cap \dot{W}^{n+1,2})(I).$ Denote its Bochner integral by
$$h=\int_{\mathbb{R}^d}h_s\,ds.$$
Then $h\in (\dot{W}^{n,2}\cap \dot{W}^{n+1,2})(I)$ and
$$\int_{\mathbb{R}^d}T^{x}_{h_s^{[n]}}(b_1,\ldots,b_n)\,ds=T^{x}_{h^{[n]}}(b_1,\ldots,b_n),$$
where the left-hand side is a Bochner integral with values in $L_\infty(\mathbb T^d_\theta)$.
\end{prop}
\begin{proof}
By Lemma \ref{lem:MOI bound Banach}, the map
$$T:(\dot{W}^{n,2}\cap \dot{W}^{n+1,2})(I)\to L_\infty(\mathbb T^d_\theta),\qquad f\mapsto T^{x}_{f^{[n]}}(b_1,\ldots,b_n)$$
is a continuous linear map between Banach spaces. Hence $s\mapsto T(h_s)$ is Bochner integrable over $\mathbb R^d$ and $\int_{\mathbb R^d}T(h_s)ds=T(\int_{\mathbb R^d} h_sds)=T(h)$.
\end{proof}

The function playing the role of $h_s$ in the above proposition will be $h_s=|s|^{-k}(\sigma_{|s|^{-2}}f)|_I$, where $(\sigma_tf)(\alpha):=f(\alpha/t)$ denotes the dilation operator.

\begin{lem}\label{bochner wiener lemma} Let $f$ be a Schwartz function on $\mathbb{R}.$ Let $I$ be a bounded open interval separated from 0. If $n\in\mathbb N$ and $k\in\mathbb Z_+$ satisfy $2n\geq k$, then the mapping
$$s\mapsto |s|^{-k}(\sigma_{|s|^{-2}}f)|_I,\quad 0\neq s\in\mathbb{R}^d.$$
is Bochner integrable to $(\dot{W}^{n,2}\cap \dot{W}^{n+1,2})(I).$
\end{lem}
\begin{proof} It is immediate that
\begin{align*}
(|s|^{-k}\sigma_{|s|^{-2}}f)^{(n)}&=|s|^{2n-k}\sigma_{|s|^{-2}}f^{(n)},\\
(|s|^{-k}\sigma_{|s|^{-2}}f)^{(n+1)}&=|s|^{2n+2-k}\sigma_{|s|^{-2}}f^{(n+1)}.
\end{align*}
These functions of $s$ are continuous from $\mathbb{R}^d\backslash\{0\}$ to $L_2(\mathbb{R}).$ Hence, the mapping $s\mapsto |s|^{-k}\sigma_{|s|^{-2}}f$ is continuous from $\mathbb{R}^d\backslash\{0\}$ to $(\dot{W}^{n,2}\cap \dot{W}^{n+1,2})(\mathbb{R}).$ Hence, the mapping $s\mapsto |s|^{-k}(\sigma_{|s|^{-2}}f)|_I$ is continuous from $\mathbb{R}^d\backslash\{0\}$ to $(\dot{W}^{n,2}\cap \dot{W}^{n+1,2})(I)$, and therefore Bochner measurable.
	
Regarding absolute integrability, we have
$$\|(|s|^{-k}\sigma_{|s|^{-2}}f)^{(n)}\|_{L_2(I)}=|s|^{2n-k-1}\|f^{(n)}\|_{L_2(|s|^2I)}\,,$$
$$\|(|s|^{-k}\sigma_{|s|^{-2}}f)^{(n+1)}\|_{L_2(I)}=|s|^{2n-k+1}\|f^{(n+1)}\|_{L_2(|s|^2I)}\,.$$
Thus,
$$\||s|^{-k}\sigma_{|s|^{-2}}f\|_{(\dot{W}^{n,2}\cap \dot{W}^{n+1,2})(I)}\leq (|s|^{2n-k-1}+|s|^{2n-k+1})\|f\|_{(\dot{W}^{n,2}\cap \dot{W}^{n+1,2})(|s|^2I)}\,.$$
As $I$ is bounded away from 0, and $f$ is Schwartz, the latter expression decays rapidly as $|s|\to\infty$. Moreover, as $2n-k\geq0$, the factor $(|s|^{2n-k-1}+|s|^{2n-k+1})$ is of order $\mathcal{O}(|s|^{-1})$ as $|s|\to0$. As $I$ is bounded, and $f^{(n)},f^{(n+1)}$ are continuous at $0$, the factor $\|f\|_{(\dot{W}^{n,2}\cap \dot{W}^{n+1,2})(|s|^2I)}$ is of order $\mathcal O(|s|)$ as $|s|\to0$. 
Hence, the mapping $s\to |s|^{-k}(\sigma_{|s|^{-2}}f)|_I$ is absolutely integrable with respect to $(\dot{W}^{n,2}\cap \dot{W}^{n+1,2})(I)$ and the assertion follows. 
\end{proof}

The following lemma gives a simplified expression for $h=\int_{\mathbb R^d}h_s\,ds$.
\begin{lem}\label{exact integral lemma} Let $f$ be a Schwartz function on $\mathbb{R}.$ Let $n\in\mathbb N$ and $k\in2\mathbb Z_+$ such that $2n\geq k.$ Let $I\subseteq(0,\infty)$ be a bounded open interval separated from $0.$ There exists a $\frac{k}{2}^\text{th}$ order primitive $F_{k,d}$ of $\alpha\mapsto\alpha^{-\frac{d}{2}}$ such that
$$\int_{\mathbb{R}^d}|s|^{-k}(\sigma_{|s|^{-2}}f)|_I\,ds=\int_{\mathbb{R}^d}f^{(\frac{k}{2})}(|s|^2)ds\cdot F_{k,d}|_I\,,$$
where the left-hand side is a Bochner integral with values in $(\dot{W}^{n,2}\cap \dot{W}^{n+1,2})(I),$ and the right-hand side is interpreted as an element of $(\dot{W}^{n,2}\cap \dot{W}^{n+1,2})(I)$ as well.
\end{lem}
\begin{proof} Note that elements of $(\dot{W}^{n,2}\cap \dot{W}^{n+1,2})(I)$ are not exactly functions, but functions modulo polynomials of degree $<n.$ For every $s\in\mathbb R^d\setminus\{0\}$, a particular representative of $|s|^{-k}(\sigma_{|s|^{-2}}f)|_I\in (\dot{W}^{n,2}\cap \dot{W}^{n+1,2})(I)$ is given by the function
$$\tilde h_s:I\to\mathbb R,\quad \tilde h_s(\alpha):=|s|^{-k}\Big(f(\alpha|s|^2)-\sum_{j=0}^{\frac{k}{2}-1}\frac{f^{(j)}(0)}{j!}(\alpha|s|^2)^{j}\cdot\phi(s)\Big).$$
Here, $\phi$ is a Schwartz function on $\mathbb{R}^d$ that equals 1 on a neighbourhood of $0.$
For a given $\alpha\in I,$ we have
$$\big|\tilde h_s(\alpha)\big|=|s|^{-k}\cdot\mathcal{O}((|s|^2)^\frac{k}{2})=\mathcal{O}(|s|^0).$$
Consequently, $s\mapsto \tilde h_s(\alpha)$ is integrable for every $\alpha\in I,$ and the same holds for $s\mapsto \tilde h_s^{(j)}(\alpha)$, $j\leq n$. Recall that $s\mapsto|s|^{-k}(\sigma_{|s|^{-2}}f)|_I$ is Bochner integrable by Lemma \ref{bochner wiener lemma}. 
By using the definition of $(\dot{W}^{n,2}\cap \dot{W}^{n+1,2})(I)$, and subsequently using dominated convergence on $\alpha\mapsto\int \tilde h_s^{(j)}(\alpha)ds$, we obtain for almost every $\alpha\in I$,
\begin{align*}
\Big(\int_{\mathbb{R}^d} |s|^{-k}(\sigma_{|s|^{-2}}f)|_I \,ds\Big)^{(n)}(\alpha)=\int_{\mathbb{R}^d} \tilde h_s^{(n)}(\alpha)\,ds
=\frac{d^n}{d\alpha^n}\Big(\int_{\mathbb{R}^d} \tilde h_s(\alpha)\,ds\Big).
\end{align*}
Therefore, $\alpha\mapsto\int \tilde h_s(\alpha)ds$ is a representative of $\int |s|^{-k}(\sigma_{|s|^{-2}}f)|_I ds$. By Lemma \ref{Fk lemma} we have
$$\int_{\mathbb{R}^d} \tilde h_s(\alpha)\,ds=\int_{\mathbb R^d}f^{(\frac{k}{2})}(|s|^2)ds\cdot F_{k,d}(\alpha)\qquad(\alpha\in I),$$
and so the proof is complete.
\end{proof}

\begin{proof}[Proof of Theorem \ref{moi integration thm}]
Let $I=(\frac{1}{2}\inf{\rm spec}(x),2\sup{\rm spec}(x))$ and $h_s=|s|^{-k}(\sigma_{|s|^{-2}}f)|_I$, $s\in\mathbb{R}^d.$ By Lemma \ref{bochner wiener lemma}, the conditions in Proposition \ref{wiener moi fact} are met for the mapping $s\mapsto h_s.$ Using Proposition \ref{wiener moi fact} (and the fact that $g\mapsto T^x_{g^{[n]}}$ factors through $g\mapsto g|_I$) we have
\begin{align*}
\int_{\mathbb{R}^d}|s|^{-k}T^{x}_{(\sigma_{|s|^{-2}}f)^{[n]}}(b_1,\ldots,b_n)\,ds&=\int_{\mathbb{R}^d}T^{x}_{h_s^{[n]}}(b_1,\ldots,b_n)\,ds\\
&=T^{x}_{h^{[n]}}(b_1,\ldots,b_n),
\end{align*}
where
$$h=\int_{\mathbb{R}^d}h_s\,ds=\int_{\mathbb{R}^d}|s|^{-k}(\sigma_{|s|^{-2}}f)|_I\,ds.$$
By Lemma \ref{exact integral lemma}, we have
$$h=\int_{\mathbb{R}^d}f^{(\frac{k}{2})}(|s|^2)ds\cdot F_{k,d}|_I,$$
completing the proof.
\end{proof}

\subsection{Integral formula relating the base cases of recursion}

We can now prove the main theorem of Section \ref{sct:Integral formula S and T}.
\begin{proof}[Proof of Theorem \ref{ik base of recursion}.] By definition of $S_{s,z}^{n}$ (equation \eqref{eq:def S^m_sz}) we have, for $s\neq0,$
\begin{align*}
&\frac1{2\pi}\int_{-\infty}^\infty|s|^{-k}S_{s,i\lambda}^{n}(b_1,\ldots,b_n)e^{i\lambda}\,d\lambda\\
&=(-1)^n|s|^{2n-k}\cdot\frac1{2\pi} \int_{-\infty}^\infty\frac{1}{x|s|^2+i\lambda}b_1\frac{1}{x|s|^2+i\lambda}\cdots b_n\frac{1}{x|s|^2+i\lambda}e^{i\lambda}\,d\lambda.
\end{align*}
Next,
$$\frac1{2\pi}\int_{-\infty}^\infty\frac{1}{x|s|^2+i\lambda}b_1\frac{1}{x|s|^2+i\lambda}\cdots b_n\frac{1}{x|s|^2+i\lambda}e^{i\lambda}d\lambda=T_{\Psi_s}^{x}(b_1,\ldots,b_n),$$
where
$$\Psi_s(\alpha_0,\ldots,\alpha_n):=\frac1{2\pi}\int_{-\infty}^\infty\frac{1}{\alpha_0|s|^2+i\lambda}\frac{1}{\alpha_1|s|^2+i\lambda}\cdots \frac{1}{\alpha_n|s|^2+i\lambda}e^{i\lambda}d\lambda,$$
for all $\alpha_0,\ldots,\alpha_n>0$. Let $f$ be any Schwartz function that on $(0,\infty)\subseteq\mathbb R$ is defined by
$$f(\alpha):=\frac1{2\pi}\int_{-\infty}^\infty\frac{1}{\alpha+i\lambda}e^{i\lambda}d\lambda=e^{-\alpha},\quad \alpha>0.$$
Using the dilation $(\sigma_{|s|^{-2}}f)(\alpha)=f(|s|^2\alpha)$ and computing the divided differences of
$\alpha\mapsto\frac{1}{\alpha|s|^2+i\lambda},$
 we obtain
$${\Psi_s}|_{(0,\infty)^{n+1}}=(-1)^n|s|^{-2n}(\sigma_{|s|^{-2}}f)^{[n]}|_{(0,\infty)^{n+1}}.$$
Therefore,
\begin{align}\label{eq:S and T}
\frac1{2\pi}\int_{-\infty}^\infty |s|^{-k}S_{s,i\lambda}^{n}(b_1,\ldots,b_n)e^{i\lambda}d\lambda=|s|^{-k}T^{x}_{(\sigma_{|s|^{-2}}f)^{[n]}}(b_1,\ldots,b_n).
\end{align}
By Theorem \ref{moi integration thm}, and the fact that
$$\int_{\mathbb R^d}f^{(\frac{k}{2})}(|s|^2)ds=(-1)^{\frac{k}{2}}\int_{\mathbb R^d}e^{-|s|^2}\,ds=(-1)^\frac{k}{2}\pi^{\frac{d}{2}},$$
the assertion follows.
\end{proof}

\section{Proof of the main theorem}\label{sct:pf main thm}
We can summarise the previous section in the following way.
\begin{cor}\label{cor:S and bT}
Let $x\in C^\infty(\mathbb T^d_\theta)$ be positive and invertible. For all $m\in\mathbb N$, $\mathbf B_1,\ldots,\mathbf B_m\in\mathcal{X}$ and $k\in 2\mathbb Z_+$ such that $2m\geq k$, we have
\begin{align*}
\frac{1}{2\pi}\int_{\mathbb{R}^d}\Big(\int_{-\infty}^\infty|s|^{-k}S^m_{s,i\lambda}(\mathbf B_1,\ldots,\mathbf B_m)e^{i\lambda}d\lambda\Big)ds=(-1)^{\frac{k}{2}}\pi^{\frac{d}{2}}\cdot \mathbf T^{x,m}_{F_{k,d}}(\mathbf B_1,\ldots,\mathbf B_m),
\end{align*}
where $F_{k,d}$ is any $\frac{k}{2}^{\text{th}}$ primitive of $\alpha\mapsto\alpha^{-\frac{d}{2}}$.
\end{cor}
\begin{proof}
This follows from the matching recursive properties of $\mathbf T^{x,m}_f$ (Definition \ref{def:bT}) and $S^m_{s,z}$ (Lemmas \ref{zeroth workhorse lemma} and \ref{first workhorse lemma}) and the base case, Theorem \ref{ik base of recursion}.
\end{proof}
The above corollary is the final ingredient needed for the proof of our main theorem.
\begin{proof}[Proof of Theorem \ref{thm: recursive formula}.]

By the definition of $I_k(P)$ as given in Section \ref{defnot section} we have
\begin{align*}
I_k(P)=&\frac{1}{2\pi}\int_{\mathbb R^d}\Big(\int_{-\infty}^\infty\sum_{\frac{k}{2}\leq m\leq k}(-1)^m\sum_{\substack{\mathscr{A}\subseteq\{1,\ldots,m\}\\|\mathscr{A}|=2m-k}}\frac{1}{x|s|^2+z}\\
&\cdot W^\mathscr{A}_1(s)\Big(\frac{1}{x|s|^2+z}\cdots W^\mathscr{A}_m(s)\Big(\frac{1}{x|s|^2+z}\Big)\cdots\Big)e^{i\lambda}\,d\lambda\Big) ds.
\end{align*}
Using the definition of $S_{s,z}^m$ (see \eqref{eq:def S^m_sz}), and introducing elements $\mathbf W_j^\mathscr{A}(s)\in\mathcal{X}$ for which $\pi(\mathbf W_j^\mathscr{A}(s))=W_j^\mathscr{A}(s)$, we rewrite the latter expression as
$$I_k(P)=\frac{1}{2\pi}\int_{\mathbb R^d}\Big(\int_{-\infty}^\infty\sum_{\frac{k}{2}\leq m\leq k}\sum_{\substack{\mathscr{A}\subseteq\{1,\ldots,m\}\\|\mathscr{A}|=2m-k}} \frac{1}{|s|^{2m}}S^m_{s,i\lambda}(\mathbf{W}_1^\mathscr{A}(s),\ldots,\mathbf{W}_m^\mathscr{A}(s))e^{i\lambda}\,d\lambda\Big) ds.$$
By expressing $\mathbf W_j^\mathscr{A}(s)$ in terms of $\mathbf W_j^{\mathscr{A},\iota}$ of \eqref{eq:bW} (see also \eqref{eq:W notation}) we obtain
$$S^m_{s,z}(\mathbf{W}_1^\mathscr{A}(s),\ldots,\mathbf{W}_m^\mathscr{A}(s))=\sum_{\iota:\mathscr{A}\to\{1,\ldots,d\}}\Big(\prod_{j\in\mathscr{A}}s_{\iota(j)}\Big)S^m_{s,z}(\mathbf{W}_1^{\mathscr{A},\iota},\ldots,\mathbf{W}_m^{\mathscr{A},\iota}).$$
Thus,
\begin{align*}
I_k(P)=&\frac{1}{2\pi}\sum_{\frac{k}{2}\leq m\leq k}\,\sum_{\substack{\mathscr{A}\subseteq\{1,\ldots,m\}\\|\mathscr{A}|=2m-k}}\,\sum_{\iota:\mathscr{A}\to\{1,\ldots,d\}}\\
&\qquad\qquad\qquad\quad\int_{\mathbb R^d}\Big(\frac{\prod_{j\in\mathscr{A}}s_{\iota(j)}}{|s|^{2m}}\int_{-\infty}^\infty S^m_{s,i\lambda}(\mathbf{W}_1^{\mathscr{A},\iota},\ldots,\mathbf{W}_m^{\mathscr{A},\iota})e^{i\lambda}\,d\lambda\Big) ds.
\end{align*}

Since the mapping
$$s\mapsto \int_{-\infty}^\infty S^m_{s,i\lambda}(\mathbf{W}_1^{\mathscr{A},\iota},\ldots,\mathbf{W}_m^{\mathscr{A},\iota})e^{i\lambda}d\lambda,\quad s\in\mathbb{R}^d,$$
is a function of $|s|,$ we can apply the general formula
\begin{align}\label{polar salpha formula}
\int_{\mathbb{R}^d}\frac{s^{n}}{|s|^{|n|_1}}g(|s|)ds&=\frac1{{\rm Vol}(\mathbb{S}^{d-1})}\int_{\mathbb{S}^{d-1}}u^{n}du\cdot \int_{\mathbb{R}^d}g(|s|)ds\\
&\equiv c_d^{(\iota)}\int_{\mathbb{R}^d}g(|s|)ds,\nonumber
\end{align}
for $n\in\mathbb Z_+^d$ satisfying $n_j=|\iota^{-1}(\{j\})|$, and find that
\begin{align*}
I_k(P)=&\sum_{\frac{k}{2}\leq m\leq k}\,\sum_{\substack{\mathscr{A}\subseteq\{1,\ldots,m\}\\|\mathscr{A}|=2m-k}}\,\sum_{\iota:\mathscr{A}\to\{1,\ldots,d\}}c_d^{(\iota)}\frac{1}{2\pi}\\
&\cdot\int_{\mathbb R^d}\Big(\int_{-\infty}^\infty |s|^{-k}S^m_{s,i\lambda}(\mathbf{W}_1^{\mathscr{A},\iota},\ldots,\mathbf{W}_m^{\mathscr{A},\iota})e^{i\lambda}\,d\lambda\Big) ds.
\end{align*}
By applying Corollary \ref{cor:S and bT}, we obtain our main theorem.
\end{proof}

\section{The case $k=2$: `scalar curvature'}\label{sct:k=2}
Although the formula in Corollary \ref{cor:main thm k=2} can be obtained directly from the computer (as done in Appendix \ref{appendix A}) we will first give an explicit proof by hand, demonstrating the simplicity of the algorithm. We fix $d\in\mathbb N_{\geq2}$ and $F_{2,d}$ as in Remark \ref{rem:Fkd}

\begin{lem}\label{short expression for I2} For any $P$ of the form \eqref{eq:P} with positive invertible $x$, \eqref{eq:main thm} gives
\begin{equation}\label{sei2 eq}
-\pi^{-\frac{d}{2}}I_2(P)=\mathbf T^{x,1}_{F_{2,d}}(\mathbf P)+\frac1d\sum_{i=1}^d\mathbf T^{x,2}_{F_{2,d}}(\mathbf A_i,\mathbf A_i).
\end{equation}
\end{lem}
\begin{proof} Our main theorem (Theorem \ref{thm: recursive formula}) in the case $k=2$ becomes
\begin{align}\label{eq:I_2 0}
I_2(P)=-\sum_{1\leq m\leq 2}\,\sum_{\substack{\mathscr A\subseteq\{1,\ldots,m\}\\|\mathscr A|=2m-2}}\,\sum_{\iota:\mathscr{A}\to\{1,\ldots,d\}} c_d^{(\iota)}\pi^{\frac{d}{2}}\mathbf T^{x,m}_{F_{2,d}}(\mathbf W^{\mathscr{A},\iota}_1,\ldots,\mathbf W^{\mathscr{A},\iota}_m).
\end{align}
If $m=1$ then $|\mathscr{A}|=2m-2$ implies $\mathscr{A}=\emptyset$; we then denote the unique function $\iota:\mathscr{A}\to\{1,\ldots,d\}$ by $\iota=\emptyset$. If $m=2$ then $|\mathscr{A}|=2m-2$ implies $\mathscr{A}=\{1,2\}$; we then identify $\iota:\{1,2\}\to\{1,\ldots,d\}$ with $(i,j)=(\iota(1),\iota(2))$ for $i,j\in\{1,\ldots,d\}$.
We obtain
\begin{align*}
I_2(P)=&-c_d^{(\emptyset)}\pi^{\frac{d}{2}}\mathbf T^{x,1}_{F_{2,d}}(\mathbf W^{\emptyset,\emptyset}_1)-\sum_{i,j=1}^d c_d^{(i,j)}\pi^{\frac{d}{2}}\mathbf T^{x,2}_{F_{2,d}}(\mathbf W^{\mathscr{A},(i,j)}_1,\mathbf W^{\mathscr{A},(i,j)}_2)\\
=&-c_d^{(\emptyset)}\pi^{\frac{d}{2}}\mathbf T^{x,1}_{F_{2,d}}(\mathbf P)-\sum_{i,j=1}^d c_d^{(i,j)}\pi^{\frac{d}{2}}\mathbf T^{x,2}_{F_{2,d}}(\mathbf A_i,\mathbf A_j).
\end{align*}
Remark \ref{rem:c_d} gives $c_d^{(\emptyset)}=\frac{(d-2)!!}{(0+d-2)!!}=1$ and $c_d^{(i,j)}=\delta_{i,j}\frac{(d-2)!!}{(2+d-2)!!}=\delta_{i,j}\frac{1}{d}$. We therefore obtain \eqref{sei2 eq}.
\end{proof}

\begin{proof}[Proof of Corollary \ref{cor:main thm k=2}]
We work out the first term on the right hand side of \eqref{sei2 eq} by using \eqref{eq:bf Ai and bf P} and Definition \ref{def:bT}. In a similar way to Example \ref{ex:bT} we obtain
\begin{align*}
\mathbf T^{x,1}_{F_{2,d}}(\mathbf P)=&\mathbf T^{x,1}_{F_{2,d}}(x\sum_{i=1}^d\mathbf D_i^2+\sum_{i=1}^da_i\mathbf D_i+a)\\
=&\sum_{i=1}^d\Big(\mathbf T^{x,2}_{F_{2,d}}(x\mathbf D_i,D_ix)+\mathbf T^{x,2}_{F_{2,d}}(a_i,D_ix)\Big)+\mathbf T^{x,1}_{F_{2,d}}(a)\\
=&\sum_{i=1}^d\Big(\mathbf T^{x,3}_{F_{2,d}}(x,D_ix,D_ix)+\mathbf T^{x,2}_{F_{2,d}}(x,D_iD_ix)+\mathbf T^{x,2}_{F_{2,d}}(x,(D_ix)\mathbf D_i)\\
&+T^{x}_{F_{2,d}^{[2]}}(a_i,D_ix)\Big)+T^{x}_{F_{2,d}^{[1]}}(a)\\
=&\sum_{i=1}^d\Big(2T^{x}_{F_{2,d}^{[3]}}(x,D_ix,D_ix)+T^{x}_{F_{2,d}^{[2]}}(a_i,D_ix)\Big)+T^{x}_{F_{2,d}^{[2]}}(x,\Delta x)+T^{x}_{F_{2,d}^{[1]}}(a).
\end{align*}
Once you get the hang of it, working out the second term is child's play. For all $i\in\{1,\ldots,d\}$ we obtain
\begin{align*}
&\mathbf T^{x,2}_{F_{2,d}}(\mathbf A_i,\mathbf A_i)\\
&=\mathbf T^{x,2}_{F_{2,d}}(2x\mathbf D_i+a_i,2x\mathbf D_i+a_i)\\
&=4\mathbf T^{x,3}_{F_{2,d}}(x\mathbf D_i,x,D_ix)+2\mathbf T^{x,2}_{F_{2,d}}(x\mathbf D_i,a_i)\\
&\quad+2T^x_{F_{2,d}^{[3]}}(a_i,x,D_ix)+T^x_{F_{2,d}^{[2]}}(a_i,a_i)\\
&=4\Big(\mathbf T^{x,4}_{F_{2,d}}(x,D_ix,x,D_ix)+\mathbf T^{x,3}_{F_{2,d}}(x,D_ix,D_ix)+\mathbf T^{x,4}_{F_{2,d}}(x,x,D_ix,D_ix)\\
&\quad+\mathbf T^{x,3}_{F_{2,d}}(x,x,D_iD_ix)+\mathbf T^{x,4}_{F_{2,d}}(x,x,D_ix,D_ix)\Big)+2\Big(\mathbf T^{x,3}_{F_{2,d}}(x,D_ix,a_i)\\
&\quad+\mathbf T^{x,2}_{F_{2,d}}(x,D_ia_i)+\mathbf T^{x,3}_{F_{2,d}}(x,a_i,D_ix)\Big)+2T^x_{F_{2,d}^{[3]}}(a_i,x,D_ix)+T^x_{F_{2,d}^{[2]}}(a_i,a_i)\\
&=4T^{x}_{F_{2,d}^{[4]}}(x,D_ix,x,D_ix)+4T^{x}_{F_{2,d}^{[3]}}(x,D_ix,D_ix)+8T^{x}_{F_{2,d}^{[4]}}(x,x,D_ix,D_ix)\\
&\quad+4T^{x}_{F_{2,d}^{[3]}}(x,x,D_i^2x)+2T^{x}_{F_{2,d}^{[3]}}(x,D_ix,a_i)+2T^{x}_{F_{2,d}^{[2]}}(x,D_ia_i)\\
&\quad+2T^{x}_{F_{2,d}^{[3]}}(x,a_i,D_ix)+2T^x_{F_{2,d}^{[3]}}(a_i,x,D_ix)+T^x_{F_{2,d}^{[2]}}(a_i,a_i).
\end{align*}
Inserting both results into \eqref{sei2 eq} yields Corollary \ref{cor:main thm k=2}.
\end{proof}

\subsection{Conjugation property}
The goal of this subsection is to show that $I_2$ satisfies the conjugation property of Proposition \ref{lem:conj property}, which is needed in the proof of Theorem \ref{cor:main thm k=d=2}.

Firstly, we define a right action of $C^\infty(\mathbb T^d_\theta)$ on $\mathcal X$ by extending
\begin{align}\label{eq:right action}
\mathbf D_i b:=D_ib+b\mathbf D_i\qquad (b\in C^\infty(\mathbb T^d_\theta)),
\end{align}
to a $C^\infty(\mathbb T^d_\theta)$-bimodule structure on $\mathcal X$ in the obvious way. (To be precise, $b\mathbf D^\alpha c=\sum_{\beta+\gamma=\alpha}bD^\beta c\mathbf D^\gamma$ if $\alpha\in\{0,1\}^d$, which one can assume without loss of generality.) One easily checks well-definedness of this structure. Moreover, using the notation introduced in \eqref{eq:def pi}, one easily checks that
\begin{align}\label{eq:pi and bimodule structure}
\pi(b\mathbf D^\alpha c)=\lambda_l(b)D^\alpha\lambda_l(c),\end{align}
for all $b,c\in C^\infty(\mathbb T^d_\theta)$ and $\alpha\in\mathbb Z_+^d$.

\begin{lem}\label{17 first lemma} Let $x,y\in C^\infty(\mathbb T^d_\theta)$ with $x$ self-adjoint, $y$ invertible, and $[x,y]=0$. We have
\begin{align*}
\mathbf{T}_{F_{2,d}}^{x,2}(y^{-1}x\mathbf{D}_iy,y^{-1}x\mathbf{D}_iy)=&y^{-1}\mathbf{T}_{F_{2,d}}^{x,2}(x\mathbf{D}_i,x\mathbf{D}_i)y+\frac12y^{-1}\cdot x^2F_{2,d}''(x)\cdot D_i^2y\\
&+y^{-1}\Big(\mathbf T^{x,2}_{F_{2,d}}(x\mathbf D_i,x)+\mathbf T^{x,2}_{F_{2,d}}(x,x\mathbf D_i)\Big)\cdot D_iy.
\end{align*}
\end{lem}
\begin{proof}

Using \eqref{eq:pi and bimodule structure} in the definition of $S_{s,z}^2$ (i.e., \eqref{eq:def S^m_sz}), it follows that
$$S^2_{s,z}(y^{-1}x\mathbf{D}_iy,y^{-1}x\mathbf{D}_iy)=y^{-1}S^2_{s,z}(x\mathbf{D}_i,x\mathbf{D}_iy).$$
Again using the definition of $S_{s,z}^2,$ we find
$$S^2_{s,z}(x\mathbf{D}_i,x\mathbf{D}_i y)=\frac{|s|^4x}{x|s|^2+z}D_i\Big(\frac{x}{x|s|^2+z}D_i\frac{y}{x|s|^2+z}\Big).$$
By the Leibniz rule, we have
\begin{align*}
D_i&\Big(\frac{x}{x|s|^2+z}D_i\frac{y}{x|s|^2+z}\Big)\\
&=D_i\Big(\frac{x}{(x|s|^2+z)^{2}}\cdot D_iy\Big)+D_i\Big(\frac{x}{x|s|^2+z}D_i\Big(\frac1{x|s|^2+z}\Big)y\Big)\\
&=\frac{x}{(x|s|^2+z)^{2}}\cdot D_i^2y+D_i\Big(\frac{x}{(x|s|^2+z)^{2}}\Big)\cdot D_iy+\\
&\quad+\frac{x}{x|s|^2+z}D_i\Big(\frac1{x|s|^2+z}\Big)\cdot D_iy+D_i\Big(\frac{x}{x|s|^2+z}D_i\Big(\frac1{x|s|^2+z}\Big)\Big)\cdot y.
\end{align*}
Again appealing to the definition \eqref{eq:def S^m_sz} of $S_{s,z}^2,$ we write
\begin{align*}
&S^2_{s,z}(x\mathbf{D}_i,x\mathbf{D}_iy)\\
&\quad=S^2_{s,z}(x,x)D_i^2y+S^2_{s,z}(x\mathbf D_i,x)D_iy+S^2_{s,z}(x,x\mathbf D_i)D_iy+S^2_{s,z}(x\mathbf D_i,x\mathbf D_i)y.
\end{align*}
By doubly integrating both sides of the above equality, applying Corollary \ref{cor:S and bT} to the resulting terms, and using that
$$T_{F_{2,d}^{[2]}}^x(x,x)=x^2\frac{1}{2}F_{2,d}''(x),$$
the lemma follows.
\end{proof}

\begin{lem}\label{17 second lemma} Let $x,y\in C^\infty(\mathbb T^d_\theta)$ be invertible with $x\geq0$ and $[x,y]=0$. We have
$$\mathbf{T}_{F_{2,d}}^{x,1}(y^{-1}x\mathbf{\Delta}y)=y^{-1}xF_{2,d}'(x)\Delta y+2\sum_{i=1}^dy^{-1}\mathbf{T}_{F_{2,d}}^{x,1}(x\mathbf D_i)D_iy+y^{-1}\mathbf T_{F_{2,d}}^{x,1}(x\mathbf{\Delta})y.$$
\end{lem}
\begin{proof} By using \eqref{eq:pi and bimodule structure} in the definition of $S_{s,z}^1,$ we find
$$S^1_{s,z}(y^{-1}x\mathbf{\Delta}y)=-\frac{|s|^2y^{-1}x}{x|s|^2+z}\Delta\Big(\frac{y}{x|s|^2+z}\Big).$$
By the Leibniz rule, we have
$$\Delta\Big(\frac{y}{x|s|^2+z}\Big)=\frac1{x|s|^2+z}\Delta y+2\sum_{i=1}^dD_i\Big(\frac1{x|s|^2+z}\Big)\cdot D_iy+\Delta\Big(\frac1{x|s|^2+z}\Big)\cdot y.$$
Thus,
$$S^1_{s,z}(y^{-1}x\mathbf{\Delta}y)=y^{-1}S^1_{s,z}(x)\Delta y+2\sum_{i=1}^dy^{-1}S^1_{s,z}(x\mathbf{D}_i) D_iy+y^{-1} S^1_{s,z}(x\mathbf{\Delta}) y.$$
Like in the previous proof, the assertion follows by appealing to Corollary \ref{cor:S and bT}.	
\end{proof}

\begin{lem}\label{17 third lemma} Let $x\in C^\infty(\mathbb T^d_\theta)$ be positive and invertible. We have
\begin{align}\label{eq:some MOIs k=2}
d\mathbf{T}_{F_{2,d}}^{x,1}(x\mathbf{D}_i)+2\mathbf T^{x,2}_{F_{2,d}}(x\mathbf D_i,x)+2\mathbf T^{x,2}_{F_{2,d}}(x,x\mathbf D_i)=0.
\end{align}
\end{lem}
\begin{proof} By the recursive definition of $\mathbf T^{x,m}_\phi$ (Definition \ref{def:bT}), we have
\begin{align}
d\mathbf{T}_{F_{2,d}}^{x,1}(x\mathbf{D}_i)&=dT_{F_{2,d}^{[2]}}^x(x,D_ix),\quad 2\mathbf T^{x,2}_{F_{2,d}}(x,x\mathbf D_i)=2T^x_{F_{2,d}^{[3]}}(x,x,D_ix),\label{eq:some MOIs k=2 1}\\
2\mathbf T^{x,2}_{F_{2,d}}(x\mathbf D_i,x)&=2\mathbf{T}^{x,3}_{F_{2,d}}(x,D_ix,x)+2\mathbf{T}^{x,2}_{F_{2,d}}(x,D_ix)+2\mathbf{T}^{x,2}_{F_{2,d}}(x,x\mathbf{D}_i)\nonumber\\
&=2T^x_{F_{2,d}^{[3]}}(x,D_ix,x)+2T^x_{F_{2,d}^{[2]}}(x,D_ix)+2T^x_{F_{2,d}^{[3]}}(x,x,D_ix)\label{eq:some MOIs k=2 2}.
\end{align}
By the definition of the multiple operator integral, we may rewrite the above terms as $T_\phi(D_ix)$, for instance, for any function $f$ we may rewrite
$$T_{f}^x(x,D_ix,x)=T^x_{\phi_1}(D_ix),\qquad\text{where}\qquad \phi_1(\alpha_0,\alpha_1)=\alpha_0f(\alpha_0,\alpha_0,\alpha_1,\alpha_1)\alpha_1.$$
In the same way, by using \eqref{eq:some MOIs k=2 1} and \eqref{eq:some MOIs k=2 2}, the left-hand side of \eqref{eq:some MOIs k=2} equals $T_{\phi}^x(D_ix),$ where
\begin{align*}
\phi(\alpha_0,\alpha_1)=&d\alpha_0F_{2,d}^{[2]}(\alpha_0,\alpha_0,\alpha_1)+2\alpha_0\alpha_1F_{2,d}^{[3]}(\alpha_0,\alpha_0,\alpha_1,\alpha_1)\\
&+2\alpha_0F_{2,d}^{[2]}(\alpha_0,\alpha_0,\alpha_1)+4\alpha_0^2F_{2,d}^{[3]}(\alpha_0,\alpha_0,\alpha_0,\alpha_1).
\end{align*}

Note that $\phi$ is homogeneous. Thus, it suffices to prove that $\phi(1,\alpha)=0.$ In other words, we need to show
$$(d+2)F_{2,d}^{[2]}(1,1,\alpha)+2\alpha F_{2,d}^{[3]}(1,1,\alpha,\alpha)+4F_{2,d}^{[3]}(1,1,1,\alpha)=0.$$
This equality is an elementary exercise, albeit rather long, and its proof is omitted.
\end{proof}

\begin{prop}\label{lem:conj property} Let $x,y\in C^\infty(\mathbb T^d_\theta)$ be invertible with $x\geq0$ and $[x,y]=0$. We have
$$I_2(\lambda_l(y^{-1}x)\Delta\lambda_l(y))=y^{-1}\cdot I_2(\lambda_l(x)\Delta)\cdot y.$$
\end{prop}
\begin{proof} From \eqref{eq:right action} we obtain
$$y^{-1}x\mathbf{\Delta} y=x\mathbf{\Delta}+\sum_{i=1}^da_i\mathbf{D}_i+a,$$	
$$a_i=2y^{-1}xD_i(y),\quad 1\leq i\leq d,\quad a=y^{-1}x\Delta y.$$
From this one can derive that 
$$\mathbf{A}_i=2y^{-1}x\mathbf{D}_iy.$$
By Lemma \ref{short expression for I2} we have
$$-\pi^{-\frac{d}{2}}I_2(\lambda_l(y^{-1}x)\Delta\lambda_l(y))=\mathbf T^{x,1}_{F_{2,d}}(y^{-1}x\mathbf{\Delta} y)+\frac4d\sum_{i=1}^d\mathbf T^{x,2}_{F_{2,d}}(y^{-1}x\mathbf{D}_iy,y^{-1}x\mathbf{D}_iy).$$
Using Lemma \ref{17 first lemma} and Lemma \ref{17 second lemma}, we write
\begin{align*}
-&\pi^{-\frac{d}{2}}I_2(\lambda_l(y^{-1}x)\Delta\lambda_l(y))\\
=&y^{-1}\Big(\mathbf T^{x,1}_{F_{2,d}}(x\mathbf{\Delta})+\frac4d\sum_{i=1}^d\mathbf T^{x,2}_{F_{2,d}}(x\mathbf{D}_i,x\mathbf{D}_i)\Big)+y^{-1}\Big(xF_{2,d}'(x)+\frac{2}{d}x^2F_{2,d}''(x)\Big)\Delta y+\\
&+\sum_{i=1}^dy^{-1}\Big(2\mathbf{T}_{F_{2,d}}^{x,1}(x\mathbf{D}_i)+\frac{4}{d}\mathbf T^{x,2}_{F_{2,d}}(x\mathbf D_i,x)+\frac4d\mathbf T^{x,2}_{F_{2,d}}(x,x\mathbf D_i)\Big) D_iy.
\end{align*}

Since $F_{2,d}$ is the primitive of $\alpha\to\alpha^{-\frac{d}{2}},$ it follows that
$$xF_{2,d}'(x)+\frac{2}{d}x^2F_{2,d}''(x)=x\cdot x^{-\frac{d}{2}}+\frac{2x^2}{d}\cdot(-\frac{d}{2}x^{-1-\frac{d}{2}})=0.$$
So, the second summand on the right hand side vanishes. Third summand on the right hand side vanishes by Lemma \ref{17 third lemma}. This completes the proof.
\end{proof}

\subsection{Proof of Theorem \ref{cor:main thm k=d=2}}

\begin{lem}\label{I2 conformal first exact} Let $x\in C^\infty(\mathbb T^d_\theta)$ be positive and invertible. Let $d\geq 2.$ We have
$$-\pi^{-\frac{d}{2}}I_2(\lambda_l(x^{\frac12})\Delta\lambda_l(x^{\frac12}))=T^x_{\Phi}(\Delta x)+\sum_{i=1}^d T^x_{\Psi}(D_ix,D_ix),$$
for the symbols $\Phi,\Psi$ defined by $F_{2,d}$, a first order primitive of $\alpha\mapsto\alpha^{-\frac{d}{2}}$, as
\begin{align*}
\Phi(\alpha_0,\alpha_1):=&\Big(\frac{\alpha_1}{\alpha_0}\Big)^{\frac12}\Big(\alpha_0F_{2,d}^{[2]}(\alpha_0,\alpha_0,\alpha_1)+\frac4d\alpha_0^2 F_{2,d}^{[3]}(\alpha_0,\alpha_0,\alpha_0,\alpha_1)\Big),\\
\Psi(\alpha_0,\alpha_1,\alpha_2):=&(\frac{\alpha_2}{\alpha_0})^{\frac12}\Big(\frac4d\alpha_0\alpha_1 F_{2,d}^{[4]}(\alpha_0,\alpha_0,\alpha_1,\alpha_1,\alpha_2)\\
&+(2+\frac4d)\alpha_0 F_{2,d}^{[3]}(\alpha_0,\alpha_0,\alpha_1,\alpha_2)+\frac8d\alpha_0^2 F_{2,d}^{[4]}(\alpha_0,\alpha_0,\alpha_0,\alpha_1,\alpha_2)\Big),
\end{align*}
for $\alpha_0,\alpha_1,\alpha_2,\alpha>0$.
\end{lem}
\begin{proof}
Using Corollary \ref{cor:main thm k=2} we find
\begin{align*}
-\pi^{-\frac{d}{2}}I_2(\lambda_l(x)\Delta)=&T^x_{F_{2,d}^{[2]}}(x,\Delta x)+\frac{4}{d}T^x_{F_{2,d}^{[3]}}(x,x,\Delta x)+\sum_{i=1}^2\bigg(
\frac4dT^x_{F_{2,d}^{[4]}}(x,D_ix,x,D_ix)\\
&+(2+\frac4d)T^x_{F_{2,d}^{[3]}}(x,D_ix,D_ix)+\frac{8}{d}T^x_{F_{2,d}^{[4]}}(x,x,D_ix,D_ix)\bigg).
\end{align*}
Combining the above formula with Proposition \ref{lem:conj property} and computing the resulting symbols yields the lemma.
\end{proof}

The rest of this section consists purely of algebraically rewriting the above formulas for $\Phi$ and $\Psi$ into a more concise form, and thusly derives Theorem \ref{cor:main thm k=d=2}.
Let us again fix $d\geq 2$ and $F_{2,d}$ as in Remark \ref{rem:Fkd}.

\begin{lem}\label{concrete Phi lemma} For $\Phi$ as in Lemma \ref{I2 conformal first exact}, we have
$$\Phi(\alpha_0,\alpha_1)=\frac{2(\alpha_0\alpha_1)^{\frac12}}{d}\cdot\frac{\alpha_0F_{2,d}^{[2]}(\alpha_0,\alpha_0,\alpha_1)-\alpha_1F_{2,d}^{[2]}(\alpha_0,\alpha_1,\alpha_1)}{\alpha_1-\alpha_0},\quad \alpha_0,\alpha_1>0.$$
\end{lem}
\begin{proof} By definition of $\Phi$ in Lemma \ref{I2 conformal first exact}, we have ($\alpha_0,\alpha_1>0$ as always)
$$(\alpha_0\alpha_1)^{-\frac12}\Phi(\alpha_0,\alpha_1)=F_{2,d}^{[2]}(\alpha_0,\alpha_0,\alpha_1)+\frac4d\alpha_0 F_{2,d}^{[3]}(\alpha_0,\alpha_0,\alpha_0,\alpha_1).$$
Thus,
\begin{align*}
&(\alpha_1-\alpha_0)\cdot (\alpha_0\alpha_1)^{-\frac12}\Phi(\alpha_0,\alpha_1)\\
&=F_{2,d}^{[1]}(\alpha_0,\alpha_1)-F_{2,d}^{[1]}(\alpha_0,\alpha_0)+\frac{4\alpha_0}{d}F_{2,d}^{[2]}(\alpha_0,\alpha_0,\alpha_1)-\frac{4\alpha_0}{d}F_{2,d}^{[2]}(\alpha_0,\alpha_0,\alpha_0).
\end{align*}
By definition,
$$F_{2,d}^{[1]}(\alpha_0,\alpha_0)+\frac{4\alpha_0}{d}F_{2,d}^{[2]}(\alpha_0,\alpha_0,\alpha_0)=F_{2,d}'(\alpha_0)+\frac{2\alpha_0}{d}F_{2,d}''(\alpha_0)=0.$$
Thus,
$$(\alpha_1-\alpha_0)\cdot (\alpha_0\alpha_1)^{-\frac12}\Phi(\alpha_0,\alpha_1)=F_{2,d}^{[1]}(\alpha_0,\alpha_1)+\frac{4\alpha_0}{d}F_{2,d}^{[2]}(\alpha_0,\alpha_0,\alpha_1)$$
and
\begin{align*}
&(\alpha_1-\alpha_0)^2\cdot (\alpha_0\alpha_1)^{-\frac12}\Phi(\alpha_0,\alpha_1)\\
&=F_{2,d}(\alpha_1)-F_{2,d}(\alpha_0)+\frac{4\alpha_0}{d}F_{2,d}^{[1]}(\alpha_0,\alpha_1)-\frac{4\alpha_0}{d}F_{2,d}^{[1]}(\alpha_0,\alpha_0)\\
&=(1-\frac{2}{d})\big(F_{2,d}(\alpha_1)-F_{2,d}(\alpha_0)\big)+\frac{2(\alpha_0+\alpha_1)}{d}F_{2,d}^{[1]}(\alpha_0,\alpha_1)-\frac{4\alpha_0}{d}F_{2,d}'(\alpha_0).
\end{align*}
The right-hand side, clearly, equals
\begin{align*}
&\frac{2}{d}\Big((\alpha_0+\alpha_1)F_{2,d}^{[1]}(\alpha_0,\alpha_1)-(\alpha_0^{1-\frac{d}{2}}+\alpha_1^{1-\frac{d}{2}})\Big)\\
&=\frac{2}{d}\Big(\alpha_0(F_{2,d}^{[1]}(\alpha_0,\alpha_1)-F_{2,d}'(\alpha_0))+\alpha_1(F_{2,d}^{[1]}(\alpha_0,\alpha_1)-F_{2,d}'(\alpha_1))\Big)\\
&=\frac2d\Big(\alpha_0(\alpha_1-\alpha_0)F_{2,d}^{[2]}(\alpha_0,\alpha_0,\alpha_1)-\alpha_1(\alpha_1-\alpha_0)F_{2,d}^{[2]}(\alpha_0,\alpha_1,\alpha_1)\Big).
\end{align*}
Thus,
\begin{align*}
(\alpha_1-\alpha_0)\cdot (\alpha_0\alpha_1)^{-\frac12}\Phi(\alpha_0,\alpha_1)=\frac2d\Big(\alpha_0F_{2,d}^{[2]}(\alpha_0,\alpha_0,\alpha_1)-\alpha_1F_{2,d}^{[2]}(\alpha_0,\alpha_1,\alpha_1)\Big).
\end{align*}

\end{proof}

\begin{lem}\label{abstract psi lemma} Let $f,g\in C^\infty((0,\infty))$ be such that $f+g'=0.$ If
$$\psi(\alpha,\beta)=f^{[2]}(\alpha,\alpha,\beta)+2g^{[3]}(\alpha,\alpha,\alpha,\beta),$$
then
$$\psi(\alpha,\beta)=-g^{[3]}(\alpha,\alpha,\beta,\beta).$$
\end{lem}
\begin{proof}
As $f=-g'$, we have
\begin{align*}
\psi(\alpha,\beta)&=-(g')^{[2]}(\alpha,\alpha,\beta)+2g^{[3]}(\alpha,\alpha,\alpha,\beta)\\
&=-\frac{(g')^{[1]}(\alpha,\beta)-(g')^{[1]}(\alpha,\alpha)}{\beta-\alpha}+2\frac{g^{[2]}(\alpha,\alpha,\beta)-\frac{1}{2}g^{(2)}(\alpha)}{\beta-\alpha}\\
&=\frac{1}{\beta-\alpha}\Big(2g^{[2]}(\alpha,\alpha,\beta)-(g')^{[1]}(\alpha,\beta)\Big)\\
&=\frac{1}{\beta-\alpha}\Big(2\frac{g^{[1]}(\alpha,\alpha)-g^{[1]}(\alpha,\beta)}{\alpha-\beta}-\frac{g^{[1]}(\alpha,\alpha)-g^{[1]}(\beta,\beta)}{\alpha-\beta}\Big)\\
&=\frac{1}{\beta-\alpha}\Big(\frac{g^{[1]}(\alpha,\alpha)-g^{[1]}(\alpha,\beta)}{\alpha-\beta}-\frac{g^{[1]}(\alpha,\beta)-g^{[1]}(\beta,\beta)}{\alpha-\beta}\Big)\\
&=\frac{1}{\beta-\alpha}(g^{[2]}(\alpha,\alpha,\beta)-g^{[2]}(\alpha,\beta,\beta)),
\end{align*}
concluding the proof.
\end{proof}

\begin{lem}\label{concrete Psi lemma} For $\Psi$ as in Lemma \ref{I2 conformal first exact}, we have
$$\Psi(\alpha,1,\beta)=-\frac{4}{d}(\alpha\beta)^{\frac12}g^{[3]}(\alpha,\alpha,\beta,\beta),\quad \alpha,\beta>0.$$
Here,
$$g(\alpha)=F_{2,d}(\alpha)+F_{2,d}^{[1]}(1,\alpha),\quad \alpha>0.$$
\end{lem}
\begin{proof} By definition of $\Psi$ in Lemma \ref{I2 conformal first exact}, we have
\begin{align*}
&\frac{d}{4}\frac{1}{(\alpha\beta)^{\frac12}}\Psi(\alpha,1,\beta)\\
&=F_{2,d}^{[4]}(\alpha,\alpha,1,1,\beta)+\Big(\frac{d}{2}+1\Big)F_{2,d}^{[3]}(\alpha,\alpha,1,\beta)+2\alpha F_{2,d}^{[4]}(\alpha,\alpha,\alpha,1,\beta).
\end{align*}
Note that
\begin{align*}
\alpha F_{2,d}^{[4]}(\alpha,\alpha,\alpha,1,\beta)&=(\alpha-1)F_{2,d}^{[4]}(\alpha,\alpha,\alpha,1,\beta)+F_{2,d}^{[4]}(\alpha,\alpha,\alpha,1,\beta)\\
&=F_{2,d}^{[3]}(\alpha,\alpha,\alpha,\beta)-F_{2,d}^{[3]}(\alpha,\alpha,1,\beta)+F_{2,d}^{[4]}(\alpha,\alpha,\alpha,1,\beta).
\end{align*}
Thus,
\begin{align*}
\frac{d}{4}\frac{1}{(\alpha\beta)^{\frac12}}\Psi(\alpha,1,\beta)&=F_{2,d}^{[4]}(\alpha,\alpha,1,1,\beta)+\Big(\frac{d}{2}-1\Big)F_{2,d}^{[3]}(\alpha,\alpha,1,\beta)\\
&\quad+2F_{2,d}^{[3]}(\alpha,\alpha,\alpha,\beta)+2F_{2,d}^{[4]}(\alpha,\alpha,\alpha,1,\beta)\\
&=f^{[2]}(\alpha,\alpha,\beta)+2g^{[3]}(\alpha,\alpha,\alpha,\beta),
\end{align*}
where
$$f(\alpha)=F_{2,d}^{[2]}(1,1,\alpha)+\Big(\frac{d}{2}-1\Big)F_{2,d}^{[1]}(1,\alpha);\qquad g(\alpha)=F_{2,d}(\alpha)+F_{2,d}^{[1]}(1,\alpha).$$
By writing out the divided differences explicitly, it is straightforward to show that $f+g'=0.$ The assertion now follows from Lemma \ref{abstract psi lemma}.
\calculations{
{\color{red} For OUR convenience, I verify below that $f+g'=0.$ This will be removed from the final text.}

Recall that $F_{2,d}(\alpha)=\alpha^{1-\frac{d}{2}}$ (modulo constant factor) --- except for $d=2,$ which case is easily separately verified. We have
$$(f+g')(\alpha)=$$
$$=\frac{\frac{\alpha^{1-\frac{d}{2}}-1}{\alpha-1}-(1-\frac{d}{2})}{\alpha-1}+(\frac{d}{2}-1)\frac{\alpha^{1-\frac{d}{2}}-1}{\alpha-1}+(1-\frac{d}{2})\alpha^{-\frac{d}{2}}+\frac{(1-\frac{d}{2})\alpha^{-\frac{d}{2}}(\alpha-1)-\alpha^{1-\frac{d}{2}}+1}{(\alpha-1)^2}=$$
$$=\frac{\alpha^{1-\frac{d}{2}}-1}{(\alpha-1)^2}-\frac{1-\frac{d}{2}}{\alpha-1}+(\frac{d}{2}-1)\frac{\alpha^{1-\frac{d}{2}}}{\alpha-1}+\frac{1-\frac{d}{2}}{\alpha-1}+(1-\frac{d}{2})\alpha^{-\frac{d}{2}}+$$
$$+(1-\frac{d}{2})\frac{\alpha^{-\frac{d}{2}}}{\alpha-1}-\frac{\alpha^{1-\frac{d}{2}}-1}{(\alpha-1)^2}=$$
$$=(\frac{d}{2}-1)\frac{\alpha^{1-\frac{d}{2}}}{\alpha-1}+(1-\frac{d}{2})\alpha^{-\frac{d}{2}}+(1-\frac{d}{2})\frac{\alpha^{-\frac{d}{2}}}{\alpha-1}=$$
$$=(\frac{d}{2}-1)\frac{\alpha^{1-\frac{d}{2}}-\alpha^{-\frac{d}{2}}}{\alpha-1}+(1-\frac{d}{2})\alpha^{-\frac{d}{2}}=(\frac{d}{2}-1)\alpha^{-\frac{d}{2}}+(1-\frac{d}{2})\alpha^{-\frac{d}{2}}=0.$$
}
\end{proof}

\begin{proof}[Proof of Theorem \ref{cor:main thm k=d=2}] By Lemma \ref{I2 conformal first exact}, we obtain the assertion of Theorem \ref{cor:main thm k=d=2}, with alternate expressions for $\Phi$ and $\Psi$. It is established in Lemma \ref{concrete Phi lemma} that  the expressions for $\Phi$  in Lemma \ref{I2 conformal first exact} and in Theorem \ref{cor:main thm k=d=2} coincide.

Note that $\Psi$ (as in Lemma \ref{I2 conformal first exact}) is homogeneous of degree $-1-\frac{d}{2}$, so that
$$\Psi(\alpha_0,\alpha_1,\alpha_2)=\alpha_1^{-1-\frac{d}{2}}\Psi(\frac{\alpha_0}{\alpha_1},1,\frac{\alpha_2}{\alpha_1}),\quad \alpha_0,\alpha_1,\alpha_2>0.$$
The required convenient expression for $\Psi$  now follows from Lemma \ref{concrete Psi lemma}.
\end{proof}

\section{Recovering the Connes--Moscovici modular curvature}
\label{sct:Connes-Moscovici}
In \cite{Connes-Moscovici}, and, independently, in \cite{FaKh1}, a formula for the so-called scalar curvature $I_2(P)$ of the conformally deformed non-commutative two-torus is given in terms of the modular operator of the corresponding conformal factors. Below we show how to recover their formulas as a special case of our result by taking $P=\lambda_l(x^{1/2})\Delta\lambda_l(x^{1/2})$, which in the notation of \cite{Connes-Moscovici} and \cite{FaKh1} corresponds to the Laplacian on functions when $\tau=i$.

For our convenience, we define the modular functional calculus as follows.
\begin{defi}\label{demistifiying fact} Let $d\in\mathbb N_{\geq2}$, $n\in\mathbb N$, and $x,V_1,\ldots,V_n\in L_\infty(\mathbb T^d_\theta)$ with $x$ positive and invertible. For any $K\in C^\infty(\mathbb R)$, $H\in C^\infty(\mathbb R^2)$, and $L\in C^\infty(\mathbb R^n)$, we set
$$K(\nabla)(V_1):=T^x_{K(\log(\frac{\alpha_1}{\alpha_0}))}(V_1),\quad H(\nabla_1,\nabla_2)(V_1,V_2):=T^x_{H(\log(\frac{\alpha_1}{\alpha_0}),\log(\frac{\alpha_2}{\alpha_1}))}(V_1,V_2),$$
$$L(\nabla_1,\ldots,\nabla_n)(V_1,\ldots,V_n):=T^x_{L(\log(\alpha_1/\alpha_0),\ldots,\log(\alpha_n/\alpha_{n-1}))}(V_1,\ldots,V_n).$$
Here and in the following, a multiple operator integral $T^x_{f(\alpha_0,\ldots,\alpha_n)}$ should be understood as $T^x_{(\alpha_0,\ldots,\alpha_n)\mapsto f(\alpha_0,\ldots,\alpha_n)}$.
\end{defi}
Consider as an important example the case that $K(\log (\alpha))=\alpha^p$. Then 
$$T^x_{K(\log(\frac{\alpha_1}{\alpha_0}))}(V_1)=T^x_{\alpha_0^{-p}\alpha_1^p}(V_1)=x^{-p}V_1x^p=\tilde\Delta^p(V_1)=K(\log\tilde\Delta)(V_1),$$
where $\tilde\Delta:V_1\mapsto e^{-h}V_1e^h$ is the modular operator corresponding to the conformal factor $h=\log x$. Extending this argument a bit further, one finds that the above definition agrees with the definitions of \cite{Connes-Moscovici,FaKh1}; see \cite{Lesch} for more clarification.

\begin{lem}\label{log diff lemma} Let $d\in\mathbb N_{\geq2}$. If $x=e^h$ for self-adjoint $h\in C^\infty(\mathbb T^d_\theta)$ then for all $K\in C^\infty(\mathbb R)$ and $H\in C^\infty(\mathbb R^2)$ we have
\begin{align*}
T^x_{\Phi_K}(\Delta x)+\sum_{i=1}^d T^x_{\Psi_{K,H}}(D_ix,D_ix)=\frac12 K(\nabla)(\Delta h)+\frac14\sum_{i=1}^dH(\nabla_1,\nabla_2)(D_ih,D_ih),
\end{align*}
for
\begin{align*}
\Phi_{K}(\alpha_0,\alpha_1)=&\frac12K(\log(\frac{\alpha_1}{\alpha_0}))\cdot\log^{[1]}(\alpha_0,\alpha_1);\\
\Psi_{K,H}(\alpha_0,\alpha_1,\alpha_2)=&K(\log(\frac{\alpha_2}{\alpha_0}))\cdot \log^{[2]}(\alpha_0,\alpha_1,\alpha_2)\\
&+\frac14H(\log(\frac{\alpha_1}{\alpha_0}),\log(\frac{\alpha_2}{\alpha_1}))\cdot\log^{[1]}(\alpha_0,\alpha_1)\cdot\log^{[1]}(\alpha_1,\alpha_2).
\end{align*}
\end{lem}
\begin{proof}
As $C^{\infty}(\mathbb{T}^d_{\theta})$ is stable under holomorphic functional calculus, we have $x=e^h\in C^{\infty}(\mathbb{T}^d_{\theta})$. 
By standard arguments in multiple operator integration theory, we find
$$D_ih=T^x_{\log^{[1]}}(D_ix);\quad \Delta h=T^x_{\log^{[1]}}(\Delta x)+2\sum_{i=1}^dT^x_{\log^{[2]}}(D_ix,D_ix).$$
Applying Definition \ref{demistifiying fact}, we obtain the lemma.
\end{proof}

It turns out that $\Phi_{K_0}=\Phi$ and $\Psi_{K_0,H_0}=\Psi$ with $\Phi,\Psi$ as in Theorem \ref{cor:main thm k=d=2} and $K_0$ and $H_0$ precisely as in the following theorem, proven in an entirely different way in \cite{Connes-Moscovici}.

\begin{thm}[Connes--Moscovici]\label{thm:Connes-Moscovici}
Let $x=e^h$ for self-adjoint $h\in C^\infty(\mathbb T^2_\theta)$ and consider $P=\lambda_l(x^{1/2})\Delta\lambda_l(x^{1/2})$ acting in $L_2(\mathbb T^2_\theta)$. We have
\begin{align*}
I_2(P)=-\frac{\pi}{2}\left(K_0(\nabla)(\Delta h)+\frac{1}{2}\sum_{i=1}^2 H_0(\nabla_1,\nabla_2)(D_ih,D_ih)\right),
\end{align*}
where
\begin{align*}
&K_0(s):=\frac{-2+s\coth\left(\frac{s}{2}\right)}{s\sinh\left(\frac{s}{2}\right)};\qquad\qquad H_0(s,t):=\\
&\frac{t(s+t)\cosh(s)-s(s+t)\cosh(t)+(s-t)(s+t+\sinh(s)+\sinh(t)-\sinh(s+t))}{st(s+t)\sinh(s/2)\sinh(t/2)\sinh^2((s+t)/2)}.
\end{align*}
\end{thm}
\begin{proof}
Let $\Phi,\Psi$ be as in Theorem \ref{cor:main thm k=d=2}.
A straightforward computation shows that
\calculations{
{\color{red} Below is the straightforward computation, which we will omit in the final text.} 
By definition (cf. Lemma \ref{I2 conformal d=2 first exact}), we have
$$\Phi(1,\alpha)=\alpha^{\frac12}\Big(\log^{[2]}(1,1,\alpha)+2\log^{[3]}(1,1,1,\alpha)\Big)=$$
$$=\alpha^{\frac12}\Big(\frac{\frac{\log(\alpha)}{\alpha-1}-1}{\alpha-1}+2\frac{\frac{\frac{\log(\alpha)}{\alpha-1}-1}{\alpha-1}+\frac12}{\alpha-1}\Big).$$
Thus,
$$\Phi(1,\alpha)\cdot\frac{\alpha-1}{\log(\alpha)}=\frac{\alpha^{\frac12}}{\log(\alpha)}\Big(\frac{\log(\alpha)}{\alpha-1}+2\frac{\frac{\log(\alpha)}{\alpha-1}-1}{\alpha-1}\Big).$$
Substituting $\alpha=e^s,$ and using $\frac{e^{s/2}}{e^s-1}=\frac{1}{2\sinh(\frac{s}{2})}$, $\frac{e^s+1}{e^s-1}=\coth(\frac{s}{2})$, we obtain
$$\Phi(1,\alpha)\cdot\frac{\alpha-1}{\log(\alpha)}=\frac{e^{\frac{s}{2}}}{s}\Big(\frac{s}{e^s-1}+2\frac{\frac{s}{e^s-1}-1}{e^s-1}\Big)=\frac1{2s\sinh(\frac{s}{2})}\Big(s+2\frac{s}{e^s-1}-2\Big)=$$
$$=\frac1{2s\sinh(\frac{s}{2})}\Big(s\frac{e^s+1}{e^s-1}-2\Big)=\frac{-2+s\coth(\frac{s}{2})}{2s\sinh(\frac{s}{2})}=\frac{1}{2}K_0(s).$$
}
$$\Phi(1,\alpha)=\frac12 K_0(\log(\alpha))\cdot\frac{\log(\alpha)}{\alpha-1},$$
for $\alpha>0$. By homogeneity, this yields
$$\Phi(\alpha_0,\alpha_1)=\frac12K_0(\log(\frac{\alpha_1}{\alpha_0}))\log^{[1]}(\alpha_0,\alpha_1)=\Phi_{K_0}(\alpha_0,\alpha_1),\quad \alpha_0,\alpha_1>0,$$
with $\Phi_{K_0}$ as in Lemma \ref{log diff lemma}.

We are now left to show that $\Psi=\Psi_{K_0,H_0}$, which by homogeneity comes down to showing that
\begin{align*}
-2(\alpha\beta)^{\frac12}g^{[3]}(\alpha,\alpha,\beta,\beta)=&K_0(\log(\beta/\alpha))\cdot \log^{[2]}(\alpha,1,\beta)\\
&+\frac14H_0(-\log(\alpha),\log(\beta))\cdot\log^{[1]}(\alpha,1)\cdot\log^{[1]}(1,\beta).
\end{align*}
By writing $-\log(\alpha)=s$ and $\log(\beta)=t$, one finds
\begin{align*}
&\frac{4}{\log^{[1]}(\alpha,1)\log^{[1]}(1,\beta)}(-2(\alpha\beta)^{\frac12}g^{[3]}(\alpha,\alpha,\beta,\beta)-K_0(\log(\beta/\alpha))\cdot \log^{[2]}(\alpha,1,\beta))\\
&=\frac{4}{\frac{\left(-st\right)}{\left(e^{-s}-1\right)\left(e^{t}-1\right)}}\Bigg(-2e^{-\frac{s}{2}}e^{\frac{t}{2}}\frac{1}{\left(e^{-s}-e^{t}\right)^{2}}\Bigg(\frac{s}{\left(e^{-s}-1\right)^{2}}+\frac{1}{e^{-s}-1}\\
&\quad-\frac{2}{e^{-s}-e^{t}}\left(\frac{e^{-s}}{e^{-s}-1}\left(-s\right)-\frac{e^{t}}{e^{t}-1}t\right)+\frac{-t}{\left(e^{t}-1\right)^{2}}+\frac{1}{e^{t}-1}\Bigg)\\
&\quad-\Bigg(-\frac{4e^{\frac{\left(t+s\right)}{2}}}{\left(t+s\right)\left(e^{\left(t+s\right)}-1\right)}+2e^{\left(\frac{t+s}{2}\right)}\left(\frac{\left(e^{\left(t+s\right)}+1\right)}{\left(e^{\left(t+s\right)}-1\right)^{2}}\right)\Bigg)\\
&\quad\cdot\left(\frac{1}{e^{-s}-e^{t}}\left(-\frac{s}{e^{-s}-1}-\frac{t}{e^{t}-1}\right)\right)\Bigg).
\end{align*}
Upon multiplying the latter expression with the denominator of $H_0(s,t)$, i.e., 
$$st(s+t)\sinh(s/2)\sinh(t/2)\sinh^2((s+t)/2),$$ and writing the result out explicitly, one straightforwardly obtains the numerator of $H_0(s,t)$, i.e., 
$$t(s+t)\cosh(s)-s(s+t)\cosh(t)+(s-t)(s+t+\sinh(s)+\sinh(t)-\sinh(s+t)).$$
This concludes the proof.
\end{proof}

\subsection{$K_0$ for general d}
We now give an example of a completely new modular formula that can be obtained from our main result, which generalises the function $K_0$ that appears in the main result of \cite{Connes-Moscovici} to any dimension. 
\calculations{
{\color{red}Best visualisation:}
https://www.desmos.com/calculator/gzz5emrh5o
}
\begin{thm}\label{thm:K0 arbitrary d}
For all $d\in\mathbb N_{\geq3}$, $x=e^h$, $h=h^*\in C^\infty(\mathbb T^d_\theta)$, $P=\lambda_l(x^{1/2})\Delta\lambda_l(x^{1/2})$ acting in $L_2(\mathbb T^d_\theta)$, we have
\begin{align}\label{eq:thm K0 arbitrary d}
I_2(P)=-\frac{\pi^{d/2}}{2}e^{(1-d/2)h}\left(K^d_0(\nabla)(\Delta h)+\frac{1}{2}\sum_{i=1}^d H^d_0(\nabla_1,\nabla_2)(D_ih,D_ih)\right),
\end{align}
for some function $H^d_0$ and
\begin{align}\label{eq:K_0^d}
K_0^d(s)=\frac{2}{d}\cdot\frac{-1-e^{(1-\frac{d}{2})s}+\frac{e^{\left(1-d/2\right)s}-1}{1-d/2}\coth\left(\frac{s}{2}\right)}{s\sinh\left(\frac{s}{2}\right)}.
\end{align}
The above formula in fact defines a function $K_0^d$ for every $d\in(2,\infty)$, satisfying
$$\lim_{d\to 2}K_0^d(s)=K_0(s),$$ with $K_0$ as in Theorem \ref{thm:Connes-Moscovici}. Moreover, $K_0^4=0$.
\end{thm}
\begin{figure}
\begin{center}
\includegraphics[scale=0.125]{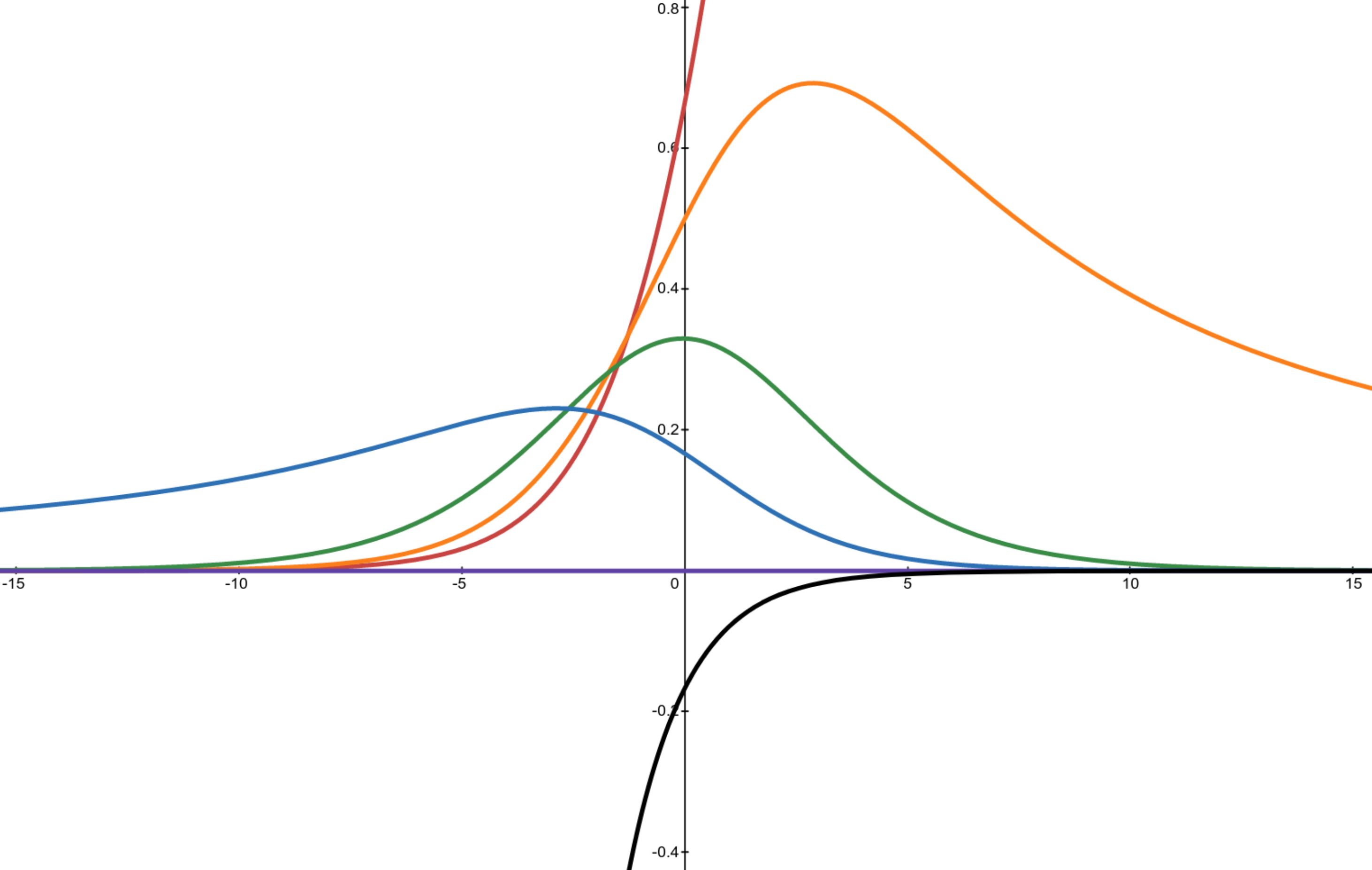}
\end{center}
\caption{The function $K_0^d$ for $d=0.01$ (red), $d=1$ (orange), $d=2.01$ (green), $d=3$ (blue), $d=4$ (purple, horizontal axis), and $d=5$ (black), plotted with Desmos.}
\end{figure}
\begin{proof}
By Theorem \ref{cor:main thm k=d=2} and Lemma \ref{log diff lemma}, we obtain \eqref{eq:thm K0 arbitrary d} exactly when $\Phi(\alpha_0,\alpha_1)=\alpha_0^{1-\frac{d}{2}}\Phi_{K_0^d}(\alpha_0,\alpha_1)$ and $\Psi(\alpha_0,\alpha_1,\alpha_2)=\alpha_0^{1-\frac{d}{2}}\Psi_{K_0^d,H_0^d}(\alpha_0,\alpha_1,\alpha_2)$. By homogeneity, we are left to derive \eqref{eq:K_0^d} for the function $K_0^d$ defined by
\begin{align}\label{eq:K_0^d in pf thm}
K_0^d(\log(\alpha)):=\frac{2}{\log^{[1]}(1,\alpha)}\cdot \Phi(1,\alpha),
\end{align}
with $\Phi$ from Theorem \ref{cor:main thm k=d=2}, namely 
\begin{align}\label{eq:Phi in pf thm K_0^d}
\Phi(1,\alpha)=\frac{2\alpha^{1/2}}{d}\frac{F_{2,d}^{[2]}(1,1,\alpha)-\alpha F_{2,d}^{[2]}(1,\alpha,\alpha)}{\alpha-1}.
\end{align}
Combining \eqref{eq:K_0^d in pf thm} with \eqref{eq:Phi in pf thm K_0^d} and substituting $\alpha=e^s$, we obtain
\calculations{
\begin{align*}
K_0^d(s)=&\frac{2(e^s-1)}{s}\cdot \frac{2e^{s/2}}{d}\Bigg(\frac{1}{(\alpha-1)^2}\Bigg(\frac{1}{1-\frac{d}{2}}\frac{\alpha^{1-\frac{d}{2}}-1}{\alpha-1}-1\Bigg)\\
&-\frac{\alpha}{(\alpha-1)^2}\Bigg(\alpha^{-\frac{d}{2}}-\frac{1}{1-\frac{d}{2}}\frac{\alpha^{1-\frac{d}{2}}-1}{\alpha-1}\Bigg)\Bigg)
\end{align*}
}
\begin{align*}
K_0^d(s)=&\frac{2(e^s-1)}{s}\cdot \frac{2e^{s/2}}{d}\Bigg(\frac{1}{(e^s-1)^2}\Bigg(\frac{1}{1-\frac{d}{2}}\frac{e^{(1-\frac{d}{2})s}-1}{e^s-1}-1\Bigg)\\
&-\frac{e^s}{(e^s-1)^2}\Bigg(e^{-\frac{d}{2}s}-\frac{1}{1-\frac{d}{2}}\frac{e^{(1-\frac{d}{2})s}-1}{e^s-1}\Bigg)\Bigg).
\end{align*}
Simplifying the above formula, one obtains the desired form of $K_0^d$. The last two statements of the theorem follow by using $\lim_{d\to2}\frac{e^{\left(1-d/2\right)s}-1}{1-d/2}=s$ and $\coth(\frac{s}{2})=\frac{1+e^{-s}}{1-e^{-s}}$, respectively.
\end{proof}
Generalisations of the functions $H_0$, $K$, $H$, $S$ \textit{et cetera} appearing in the main result \cite[Theorem 3.2]{Connes-Moscovici} can be similarly obtained. In fact, similar formulas for $k>2$ are now within easy reach. 
Although we have proved the theorem above as a consequence of Theorem \ref{cor:main thm k=d=2} (which we think is interesting in its own right) we stress that it can also be obtained directly from Theorem \ref{thm: recursive formula}, with the only difference that the intermediate formulas become longer.
In this way one can obtain any function in the modular operator, for any $P$ and $k$ one chooses.

\section{The relation with the Iochum--Masson-approach for rational $\theta$}\label{sct:Iochum-Masson}
We now relate our approach with the one of \cite{IM1,IM2,IM3}, in which Iochum and Masson calculate the local invariants for differential operators on finite dimensional bundles over manifolds.
Suppose that $\theta\in M_d(\mathbb R)$ is such that we can identify
$$C^\infty(\mathbb T_\theta^d)\subseteq M_N(C^\infty(\mathbb T^d)),$$
where $M_N(C^\infty(\mathbb T^d))$ is the algebra of $N\times N$ matrices with entries in $C^\infty(\mathbb T^d)$. There is such an inclusion for $d=2$ and rational $\theta$, or in higher dimensions under a slightly more convoluted condition on the entries of $\theta\in M_d(\mathbb R)$ (cf. \cite{Rieffel1990}). In these cases the results of \cite{IM1,IM2,IM3} can be applied. In \cite[Appendix B]{IM2}, the final result of \cite{IM2} is compared to the final result of \cite{FaKh1}, for $k=2$ and $d=2$. Here, we compare our result to \cite{IM1} for any $d$ and any $k$. 

For all $m\in\mathbb N$, $\xi\in\mathbb R^d$ and all matrix-valued differential operators $B_1,\ldots,B_m$, \cite[eq. (2.1)]{IM1} defines a matrix denoted $f_m(\xi)[B_1\otimes\cdots\otimes B_m]\in M_N(\mathbb C)$, by setting, for all $v\in \mathbb C^N$,
\begin{align}\label{eq:Iochum-Masson functional calc}
&f_m(\xi)[B_1\otimes\cdots\otimes B_m]v\\
&:=\int_{\Delta_m}e^{(s_1-1)|\xi|^2\lambda_l(x)}B_1e^{(s_2-s_1)|\xi|^2\lambda_l(x)}\cdots B_me^{(s_{m+1}-s_m)|\xi|^2\lambda_l(x)}(1_M\otimes v)\,ds,\nonumber
\end{align}
where $(1_M\otimes v)$ is just the section in $C^\infty(\mathbb T^d;\mathbb C^N))$ that is constantly $v$, and $\Delta_m=\{s\in\mathbb R^m_+:~0\leq s_m\leq\cdots\leq s_1\leq 1\}$ is the simplex equipped with the flat measure $ds$ of total variation $1/m!$, and $s_{m+1}:=0$. The integrand in \eqref{eq:Iochum-Masson functional calc} can be identified with a $\mathbb C^N$-valued function on $\mathbb T^d$ for every $\xi$ and $s$, which for every $\xi$ is Bochner-integrable in $s$.

With respect to the notation of \cite{IM1}, we restrict ourselves to $g^{\mu\nu}=\delta_{\mu\nu}$,
and substitute $u^{\mu\nu}=x\delta_{\mu\nu}$, $v^\mu(\cdot)=ia_\mu$, $w=-a$, and ${\rm tr}=(2\pi)^d\tau$ into the formulas \cite[eqs. (1.6-1.7)]{IM1}, and note that we may identify $-i\partial_j=D_j$.
In this case, the formulas \cite[(2.4-2.5), etc.]{IM1} state that the local invariants of order $k\in\{0,2,4\}$ are given by
\begin{align}\label{eq:Iochum-Masson}
I_k(P)=\int_{\mathbb R^d}\sum_{\frac k2\leq m\leq k}(-1)^m\sum_{\substack{\mathscr A\subseteq \{1,\ldots,m\}\\|\mathscr A|=2m-k}}f_m(\xi)[W_1^{\mathscr A}(\xi)\otimes\cdots\otimes W_m^{\mathscr A}(\xi)]\, d\xi,
\end{align}
in which we use the notation $W^\mathscr{A}_i$ from \eqref{eq:W notation}.
Clearly the formula \eqref{eq:Iochum-Masson} works for any $k$, which Iochum and Masson have also noted in \cite{IM202X}.

We can express their `functional calculus' \eqref{eq:Iochum-Masson functional calc} as a multiple operator integral.
\begin{lem}\label{lem:1st calc for coincidence with Iochum-Masson} For all elements $b_1\ldots,b_m\in C^\infty(\mathbb T^d_\theta)$ we have
\begin{align}
&(-1)^m f_m(\xi)[b_1\otimes\cdots\otimes b_m]=
T^{|\xi|^2 x}_{g^{[m]}}(b_1,\ldots,b_m),
\end{align}
where $g(\alpha):=e^{-\alpha}$.
\end{lem}
\begin{proof}
We obtain,
\begin{align}\label{eq:1st calc for coincidence with Iochum-Masson}
&(-1)^m f_m(\xi)[b_1\otimes\cdots\otimes b_m]v\nonumber\\
&=(-1)^m\int_{\Delta_m}e^{(s_1-1)|\xi|^2x}b_1e^{(s_2-s_1)|\xi|^2x}\cdots b_me^{-s_m|\xi|^2x}v\,ds.
\end{align}
As $b_1,\ldots,b_m$ are bounded operators, the above expression is a multiple operator integral for which we can compute the symbol inductively. Indeed, by using
$$\frac{e^{-s_m\alpha_m}-e^{-s_m\alpha_{m+1}}}{\alpha_m-\alpha_{m+1}}=-\int_0^{s_m}e^{(s_{m+1}-s_m)\alpha_m}e^{-s_{m+1}\alpha_{m+1}}\,ds_{m+1}.$$
we find
\begin{align*}
(-1)^m\int_{\Delta_m}e^{(s_1-1)\alpha_0}e^{(s_2-s_1)\alpha_1}\cdots e^{-s_m\alpha_m}\,ds=g^{[m]}(\alpha_0,\ldots,\alpha_m),
\end{align*}
which is the symbol of the multiple operator integral of \eqref{eq:1st calc for coincidence with Iochum-Masson}.
\end{proof}

\begin{prop}\label{prop:IM operators}
For all $b_1,\ldots,b_m\in C^\infty(\mathbb T^d_\theta)$ we have
\begin{align}\label{eq:IM operators}
\frac{1}{2\pi}|\xi|^{-2m}\int_{-\infty}^\infty S^m_{\xi,i\lambda}(b_1,\ldots,b_m)e^{i\lambda}\,d\lambda=(-1)^m f_m(\xi)[b_1\otimes\cdots\otimes b_m].
\end{align}
More generally, for all $\mathbf B_1,\ldots,\mathbf B_m\in \mathcal{X}$ we have
\begin{align*}
\frac{1}{2\pi}|\xi|^{-2m}\int_{-\infty}^\infty S^m_{\xi,i\lambda}(\mathbf B_1,\ldots,\mathbf B_m)e^{i\lambda}\,d\lambda=(-1)^{m} f_m(\xi)[\pi(\mathbf B_1)\otimes\cdots\otimes \pi(\mathbf B_m)].
\end{align*}
\end{prop}
\begin{proof}
By using change of variables in the definition of the multiple operator integral, and subsequently using Lemma \ref{lem:1st calc for coincidence with Iochum-Masson}, we obtain
\begin{align*}
|\xi|^{-2m}T^{x}_{(\sigma_{|\xi|^{-2}}g)^{[m]}}(b_1,\ldots,b_m)&=T^{|\xi|^2 x}_{g^{[m]}}(b_1,\ldots,b_m)\\
&=(-1)^m f_m(\xi)[b_1\otimes\cdots\otimes b_m].
\end{align*}
By applying \eqref{eq:S and T} we find the first part of the proposition.

When replacing $b_j$ in \eqref{eq:IM operators} by differential operators, the left-hand side satisfies the same recursive properties as the right-hand side: compare \cite[Lemma 2.1]{IM1} with our Lemma \ref{zeroth workhorse lemma} and Lemma \ref{first workhorse lemma}. By induction (in which \eqref{eq:IM operators} is the induction base) the second part of the proposition follows.
\end{proof} 

\begin{thm}\label{thm:IM}
Let $\theta$ be such that $C(\mathbb T_\theta^d)\subseteq M_N(C(\mathbb T^d))$. 
With the notations as above, we have, for all $\frac{k}{2}\leq m\leq k$ and $\mathscr{A}\subseteq\{1,\ldots,d\}$ with $|\mathscr A|=2m-k$,
\begin{align}\label{eq:IM thm}
\nonumber &(-1)^{\frac{k}{2}}\pi^{\frac{d}{2}}\sum_{\iota:\mathscr{A}\to\{1,\ldots,d\}} c_d^{(\iota)}\, \mathbf T^{x,m}_{F_{k,d}}(\mathbf W_1^{\mathscr{A},\iota},\ldots,\mathbf W_m^{\mathscr{A},\iota})\\
&=\int_{\mathbb R^d}(-1)^mf_m(\xi)[W_1^{\mathscr A}(\xi)\otimes\cdots\otimes W_m^{\mathscr A}(\xi)]\, d\xi,
\end{align}
and our main result (Theorem \ref{thm: recursive formula}) therefore reproduces \eqref{eq:Iochum-Masson} due to \cite{IM1}.
\end{thm}
\begin{proof}
By applying Corollary \ref{cor:S and bT} and, subsequently, Proposition \ref{prop:IM operators}, we find
\begin{align}\label{eq:IM thm 1}
\nonumber&(-1)^{\frac{k}{2}}\pi^{\frac{d}{2}}\sum_{\iota:\mathscr{A}\to\{1,\ldots,d\}} c_d^{(\iota)}\, \mathbf T^{x,m}_{F_{k,d}}(\mathbf W_1^{\mathscr{A},\iota},\ldots,\mathbf W_m^{\mathscr{A},\iota})\\
\nonumber&=\sum_{\iota:\mathscr{A}\to\{1,\ldots,d\}} c_d^{(\iota)}\frac{1}{2\pi}\int_{\mathbb R^d}|\xi|^{-k}\int_{-\infty}^\infty S^m_{\xi,i\lambda}(\mathbf W_1^{\mathscr{A},\iota},\ldots,\mathbf W_m^{\mathscr{A},\iota})e^{i\lambda}\,d\lambda\,d\xi\\
&=(-1)^m\sum_{\iota:\mathscr{A}\to\{1,\ldots,d\}} c_d^{(\iota)}\int_{\mathbb R^d}|\xi|^{2m-k}f_m(\xi)[\pi(\mathbf W_1^{\mathscr{A},\iota})\otimes\cdots\otimes \pi(\mathbf W^{\mathscr{A},\iota}_m)]\,d\xi.
\end{align}

Turning our attention to the right-hand side of \eqref{eq:IM thm}, we express $W_j^{\mathscr{A}}(\xi)$ in terms of $W_j^{\mathscr{A},\iota}:=\pi(\mathbf W_j^{\mathscr{A},\iota})$ (see \eqref{eq:W notation} and \eqref{eq:bW}), and find
$$f_m[W_1^\mathscr{A}(\xi)\otimes\cdots\otimes W_m^{\mathscr{A}}(\xi)]=\sum_{\iota:\mathscr{A}\to\{1,\ldots,d\}}\prod_{j\in\mathscr{A}}\xi_{\iota(j)}f_m[W_1^{\mathscr{A},\iota}\otimes \cdots\otimes W_m^{\mathscr{A},\iota}].$$
Hence, by using \eqref{polar salpha formula},
\begin{align}\label{eq:IM thm 2}
\nonumber&(-1)^m\int_{\mathbb R^d} f_m(\xi)[W_1^\mathscr{A}(\xi)\otimes\cdots\otimes W_m^\mathscr{A}(\xi)]\,d\xi\\
&=(-1)^m\sum_{\iota:\mathscr{A}\to\{1,\ldots,d\}} c_d^{(\iota)}\int_{\mathbb R^d}|\xi|^{2m-k}f_m(\xi)[W_1^{\mathscr{A},\iota}\otimes\cdots\otimes W_m^{\mathscr{A},\iota}]\,d\xi
\end{align}
The theorem follows by combining \eqref{eq:IM thm 1} with \eqref{eq:IM thm 2}.
\end{proof}



\appendix
\section{Comments on the accompanying python program}\label{appendix A}
Accompanying this paper is a python script that computes $I_k=I_k(\lambda_l(x)\Delta+\sum_i\lambda_l(a_i)D_i+\lambda_l(a))$ for any $k\in \mathbb Z_+$ in terms of multiple operator integrals with arguments in $C^\infty(\mathbb T^d_\theta)$. The program can also be found on the Github page \url{https://github.com/TDHvanNuland/I_k}.

The program outputs an identity (formatted in latex) with $I_k$ on the left-hand side and an expression on the right-hand side with explicit dependency on $d$. The program can be easily adjusted to match the output type one prefers. 

The last line of the program fixes the value of $k$.
For example, to compute $I_2$ one can replace the last line \begin{verbatim}
print_I(4)
\end{verbatim}
with the line
\begin{verbatim}
print_I(2)
\end{verbatim}
and run the file, for example by opening a terminal, navigating to the correct directory, and typing
\begin{verbatim}
python3 I_k.py
\end{verbatim}
(any installed version of python should work). The output should be as follows.
{\tiny
$$-\pi^{-d/2}I_2=\sum_{i}\Bigg(
2T^x_{F_{2,d}^{[3]}}(x,D_ix,D_ix)+T^x_{F_{2,d}^{[2]}}(x,D_iD_ix)+T^x_{F_{2,d}^{[2]}}(a_i,D_ix)+T^x_{F_{2,d}^{[1]}}(a/d)\Bigg)$$
$$+\sum_{i}\frac{1}{d}\Bigg(
4T^x_{F_{2,d}^{[4]}}(x,D_ix,x,D_ix)+4T^x_{F_{2,d}^{[3]}}(x,D_ix,D_ix)+8T^x_{F_{2,d}^{[4]}}(x,x,D_ix,D_ix)$$
$$+4T^x_{F_{2,d}^{[3]}}(x,x,D_iD_ix)+2T^x_{F_{2,d}^{[3]}}(x,D_ix,a_i)+2T^x_{F_{2,d}^{[2]}}(x,D_ia_i)$$
$$+2T^x_{F_{2,d}^{[3]}}(x,a_i,D_ix)+2T^x_{F_{2,d}^{[3]}}(a_i,x,D_ix)+T^x_{F_{2,d}^{[2]}}(a_i,a_i)\Bigg)$$
}
Computing $I_4$ in this way produces the 1046 terms in an instant, and computing $I_6$ takes about 10 minutes on an Intel(R) Core(T) i9-10900 CPU @ 2.80GHz, and produces 140845 terms (of course excluding sums over the indices $i,j,k,\ldots$, as otherwise the amount of terms depends on $d$).

\end{document}